\documentclass[12pt]{amsart}  

\usepackage{bm}
\usepackage{graphicx,verbatim}
\usepackage{amssymb}
\usepackage{amsthm}
\usepackage{amsmath}
\usepackage{bbm}
\usepackage{epstopdf}
\usepackage{textgreek}
\usepackage{amsfonts}
\usepackage{mathrsfs}
\usepackage{lscape} 
\usepackage[vcentermath,enableskew]{youngtab}
\usepackage{tikz,tikz-cd}
\usetikzlibrary{arrows,positioning,automata,shadows,fit,shapes}
\usepackage[english]{babel}
\usepackage{setspace}
\usepackage{microtype}

\usepackage{thmtools}

\usepackage{dynkin-diagrams}
\usepackage{tabu}
\usepackage{chngcntr}
\counterwithin{table}{section}

\newtheorem{theorem}{Theorem}[section]
\newtheorem{definition}[theorem]{Definition}
\newtheorem{example}[theorem]{Example}
\newtheorem{corollary}[theorem]{Corollary}
\newtheorem{lemma}[theorem]{Lemma}

\newtheorem{remark}[theorem]{Remark}
\newtheorem*{theorem*}{Theorem}
\newtheorem*{lemma*}{Lemma}
\newtheorem*{definition*}{Definition}

\DeclareMathOperator{\triv}{triv}

\DeclareMathOperator{\Det}{det}
\DeclareMathOperator{\Dim}{dim}

\DeclareMathOperator{\End}{End}

\DeclareMathOperator{\trace}{trace}
\DeclareMathOperator{\str}{str}
\DeclareMathOperator{\sgn}{sgn}

\DeclareMathOperator{\Ind}{Ind}

\DeclareMathOperator{\Hom}{Hom} 
\DeclareMathOperator{\Span}{span}
\DeclareMathOperator{\pr}{pr}
\DeclareMathOperator{\rad}{rad}

\DeclareMathOperator{\Lie}{Lie}

\DeclareMathOperator{\Id}{Id}

\title[Compact Schur-Weyl duality]{Compact Schur-Weyl duality and the Type B/C  VW-algebra}

\author{Kieran Calvert}             
\thanks{Department of Mathematics, University of Manchester, kieran.calvert@manchester.ac.uk}

\begin{document}

\begin{abstract} We define an extension of the VW-algebra, the type $B/C$ VW-algebra. This new algebra contains the hyperoctahedral group and it naturally acts on $\End_K(X \otimes V^{\otimes k})$ for Orthogonal and Symplectic groups. Thus we obtain a compact analogue of Schur-Weyl duality. We study functors $F_{\mu,k}$ from the category of admissible $O(p,q)$ or $Sp_{2n}(\mathbb{R})$ modules to representations of the type $B/C$ VW-algebra $\mathfrak{B}_k^\theta$. Thus providing a Akawaka-Suzuki-esque link between $O(p,q)$ (or $Sp_{2n}(\mathbb{R})$) and $\mathfrak{B}_k^\theta$.  Furthermore these functors take non spherical principal series modules to principal series modules for the graded Hecke algebra of type $D_k$, $C_{n-k}$ or $B_{n-k}$. With this we get a functorial correspondence between admissible simple $O(p,q)$ (or $Sp_{2n}(\mathbb{R})$) modules and graded Hecke algebra modules.  \end{abstract}

\maketitle                  

\tableofcontents            

\begin{section}{Introduction}
Let $G$ be an odd real orthogonal group or symplectic group, $G$ is $O(p,q)$ for $p +q=2n+1$ or $Sp_{2n}(\mathbb{R})$. Let $K$ denote a maximal compact subgroup of $G$. Let $\mathfrak{g}_0$ be the real Lie algebra of $G$. Define its complexification $\mathfrak{g} = \mathfrak{g}_0 \otimes_\mathbb{R} \mathbb{C}$.  Let $X$ be an admissible $G$-module and let $V$ be the defining matrix module of the linear group $G$.  The papers \cite{AS98,We88,CDM06,BS12,ES14,ES16} study variants of the $\mathbb{C}$-algebra $\End_G(X \otimes V^{\otimes k})$ of operators on $X \otimes V^{\otimes k}$ commuting with $G$.  For $G= Sp_{2n}(\mathbb{R}), O(p,q)$ there is a homomorphism  of the VW-algebra or degenerate BMW algebra \cite{ES14,DRV12} to $\End_G( X \otimes V^{\otimes k})$.  In this paper we focus on the larger algebra of operators which commute with $K$: 
$$\End_K(X \otimes V^{\otimes k}).$$
We define an extension of the VW-algebra, $\mathfrak{B}_k^\theta$, by operators related to the Cartan involution $\theta$ of $G$. This new algebra $\mathfrak{B}_k^\theta$ acts on $X \otimes V^{\otimes k}$ and commutes with the action of $K$. It is also an extension of the Cyclotomic Brauer algebra, which is unsurprising since the Author has shown \cite{C20} that the cyclotomic Brauer algebra acts on $\End_K(V^{\otimes k})$. It is the analogue of the VW-algebra for operators commuting with $K$. The extension contains the Weyl group of type $B /C$, the hyperoctahedral group. 
This new algebra's module category is a natural image for the functors defined by Ciubotaru and Trapa \cite{CT11}:
 $$F_{\mu,k}(X) = \Hom_K(\mu, X\otimes V^{\otimes k}).$$
 We show that the functors $F_{\mu,k}$ take the category of admissible $O(p,q)$ or $Sp_{2n}(\mathbb{R})$-modules to $\mathfrak{B}_k^\theta$-modules. Unlike previous functors, for $G =O(p,q)$ or $ Sp_{2n}(\mathbb{R})$, both categories are related to the hyperoctahedral group. 
Let $G = KAN$ be the Iwasawa decomposition of $G$, and $P = MAN$ be the minimal parabolic subgroup. For  characters $\delta$ of $M$ and $e^\nu$ of $A$, the principal series representation $X_{\delta}^\nu$ (Definition \ref{principalseriesdef}) is:
$$X_{\delta}^\nu = \Ind_{MAN}^G (\delta \otimes e^\nu \otimes 1).$$
For split real orthogonal or symplectic groups this covers all of the principal series modules. 
When $G =O(p,q)$ or $ Sp_{2n}(\mathbb{R})$ then $M = (\mathbb{Z}_2)^n$ or $O(p-q) \times (\mathbb{Z}_2)^q$. Denote the character of $M$ which is $triv$ (or $\det$) on $O(p-q)$, $-1$ on the first $k$ generators and $1$ on the remaining $n-k$ or $q-k$ by $\delta^k_{\triv}$ (resp. $\delta^k_{\det}$). For $Sp_{2n}(\mathbb{R})$ we drop the subscript $\det$ and $\triv$. The graded Hecke algebra $\mathbb{H}_k(c)$ (Definition \ref{Heckedef}) is the graded Hecke algebra associated to the hyperoctahedral group $W(B_k)$ with a certain parameter function related to $c \in \mathbb{R}$. 
For $G = Sp_{2n}(\mathbb{R})$, the functors $F_{\triv,k}$ and $F_{\det,n-k}$ take principal series modules $X_{\delta^k}^\nu$ to principal series modules for the graded Hecke algebra $\mathbb{H}_k(0)$ and $\mathbb{H}_{n-k}(1)$ respectively. 
For $G= O(p,q)$ the functors $F_{\triv \otimes \det,k}$ and $F_{\det \otimes \triv,q-k}$ take principal series modules $X_{\delta^k_{\triv}}^\nu$ to principal series modules for the graded Hecke algebra $\mathbb{H}_k(0)$ and $\mathbb{H}_{q-k}(1)$ respectively. A similar result holds for $X_{\delta^k_{\det}}^\nu$ and functors $F_{\triv \otimes \triv,k}$ and $F_{\det \otimes \det,q-k}$.
Given a particular character $\delta$ of $M$ we associate to it $K$-characters $\mu$, and $\underline{\mu}$ (Table \ref{muvalues}) with scalars $c_\mu$ and $c_{\underline{\mu}}$  (Table \ref{muconstants}).
We prove that functors related to $\mu$ and $c_\mu$ take principal series representations to principal series representations.  Thus we have defined a link between principal series of split real orthogonal or symplectic groups with principal series of certain graded Hecke algebras.

\begin{restatable*}{theorem2}{principaltoprincipalthm}For $G = Sp_{2n}(\mathbb{R})$ or $O(p,q)$ $p+q = 2n+1$, the module $F_{\mu,k}( X_{\delta}^\nu)$  is isomorphic to the $\mathbb{H}_k(c_\mu)$ principal series module
$$X(\nu_k) = \Ind_{S(\mathfrak{a}_k)}^{\mathbb{H}_k(c_\mu)} \nu_k.$$
The module $F_{\underline{\mu},n-k}( X_{\delta}^\nu)$  is isomorphic to the $\mathbb{H}_{n-k}(c_{\underline{\mu}})$ principal series module
$$X(\bar{\nu}_{n-k}) = \Ind_{S(\bar{\mathfrak{a}}_{n-k})}^{\mathbb{H}_k(c_{\underline{\mu}})} \bar{\nu}_{n-k}.$$
\end{restatable*}

 This extends the results of Ciubotaru and Trapa \cite{CT11} to non-spherical principal series modules. Importantly, if $G$ is a split real orthogonal or symplectic group, we can describe the Hecke algebra module of the image of every principal series modules resulting from functors $F_{\mu,k}$ and $F_{\underline{\mu},n-k}$. Furthermore using Casselman's subrepresention theorem, for these split groups we have a correspondence of irreducible Harish-Chandra modules of $G$ and graded Hecke algebra modules. 
\begin{restatable*}{theorem2}{splitthm} 
Let $G$ be $O(n+1,n)$ or $Sp_{2n}(\mathbb{R})$, then $G$ is split. Let $X$ be an irreducible $G$-module. Let $X_\delta^\nu$ be a principal series representation that contains $X$, then the $\mathfrak{B}_k^\theta$ and $\mathfrak{B}_{n-k}^\theta$-modules 
$$F_{\mu,k}(X) \text{ and } F_{\underline{\mu},n-k}(X)$$
are submodules of the $\mathbb{H}_k(c_\mu)$ and $\mathbb{H}_{n-k}(c_{\underline{\mu}})$-modules
$$X(\nu_k) \text{ and } X(\bar{\nu}_{n-k}).$$ 
\end{restatable*}

We define two anti-involutions on $\mathfrak{B}_k^\theta$ which descend to the usual anti-involutions on the graded Hecke algebra \cite{BC15}. Furthermore we show that if $X$ is a Hermitian (resp. unitary) module of $G= Sp_{2n}(\mathbb{R})$ then the image of $X$ under the functor $F_{\mu,k}$ is a Hermitian (resp. unitary) module for $\mathfrak{B}_k^\theta[m_0,m_1]$. 
We also show that the Langlands quotient is preserved.
\begin{restatable*}{theorem2}{langlandsquotientthm} \label{langlandsquotientpreserved} Let $X_{\delta}^\nu$ be a principal series module for $G =O(p,q)$ or $Sp_{2n}(\mathbb{R})$. The Langlands quotient $\overline{X_{\delta}^\nu} = X_{\delta}^\nu/\rad\langle ,\rangle_{X_{\delta}^\nu}$ is mapped by $F_{\mu,k}$, to the Langlands quotient of the $\mathbb{H}_k(c_\mu)$-module, $\overline{X(\nu_k)} = X(\nu_k) /\rad\langle ,\rangle_{X_{\nu_k}}$. Similarly, $X_{\delta}^\nu$ is mapped by $F_{\underline{\mu},n-k}$, to the $\mathbb{H}_{n-k}(c_{\underline{\mu}})$-module $\overline{X(\bar{\nu}_{n-k})}$.
\end{restatable*}

We then give a non-unitary test for principal series modules.

\begin{restatable*}{theorem2}{nonunitarytestthm}[Non-unitary test for principal series modules] If either $\overline{X(\nu_k)}$ or $\overline{X(\bar{\nu}_{n-k})}$ are not unitary, as $\mathbb{H}_k(c_\mu)$ and $\mathbb{H}_{n-k}(c_{\underline{\mu}})$-modules, then the Langlands quotient of the minimal principal series  module $\overline{X_{\delta^k}^\nu}$, for $G = O(p,q)$ or $ Sp_{2n}(\mathbb{R})$ is not unitary.   
\end{restatable*}

This result gives a functorial result similar to the nonunitarity criterion proved by Barbasch, Pantano, Paul and Salamanca-Riba \cite{BP04,PPS10}.
We also obtain a non-unitary test for any Harish Chandra module; in the split case one could check unitarity of Hecke algebra modules however in the non-split case one would have to work with type $B/C$ Brauer algebra modules.

\begin{restatable*}{theorem2}{nonunitarybrauerthm}[Non-unitary test for Harish-Chandra modules] Let $X$ be a Harish Chandra module. 
For $G = Sp_{2n}(\mathbb{R})$ or $O(p,q)$ $p+q = 2n+1$, if for any character $\mu$ and $k  =1,..,n$ the $\mathfrak{B}_k^\theta$-module $\overline{F_{\mu,k}(X) }$ is not unitary, then the Langlands quotient $\overline{X}$ of $X$ is not a  unitary $G$-module. 
In the case when $G$ is split then $X$  is a subrepresentation of $X_\delta^\nu$ and $F_{\mu,k}(X)$, $F_{\underline{\mu},n-k}(X)$ are Hecke algebra modules. In this case, if either $\overline{F_{\mu,k}(X)}$, $\overline{F_{\underline{\mu},n-k}(X)}$ is not unitary as a Hecke algebra module then $\overline{X}$ is not unitary as a $G$-module.
\end{restatable*}

In Section \ref{brauer}, we define the type $B/C$ VW-algebra $\mathfrak{B}_k^\theta$ and show that it acts on $X \otimes V^{\otimes k}$ and commutes with the action of $K$.  Section \ref{quotientsofbrauer} defines particular quotients of $\mathfrak{B}_k^\theta$ isomorphic to the graded Hecke algebras $\mathbb{H}_k(c)$. In Section \ref{functorssection}, we introduce the functors, defined in \cite{CT11}, $F_{\mu,k}: \mathcal{HC}(G) \to \mathfrak{B}_n^\theta$-mod. These functors naturally create $\mathfrak{B}_k^\theta$-modules. In Section \ref{restrictedfunctors}, we show that the functors restricted to principal series modules define Hecke algebra modules. Section \ref{isoclassofX} describes the isomorphism classes of $F_{\mu,k}(X_{\delta^k}^\nu)$ and   $F_{\underline{\mu},n-k}(X_{\delta^k}^\nu)$ as principal series modules of graded Hecke algebras $\mathbb{H}_k(c_\mu)$ and $\mathbb{H}_{n-k}(c_{\underline{\mu}})$. In Section \ref{hermitian}, we prove that functors $F_{\mu,k}$ preserve unitarity and invariant Hermitian forms.

\end{section}

\begin{section}{Preliminaries}
Throughout this paper we fix the following notation.
Let $G$ be $O(p,q)$, $p+q=2n+1$ or $Sp_{2n}(\mathbb{R})$. Let $\mathfrak{g}_0$ be its Lie algebra, with complexification $\mathfrak{g}= \mathfrak{g}_0\otimes_\mathbb{R} \mathbb{C}$. We uniformly denote a real Lie algebra by $\mathfrak{g}_0$, for a complex Lie algebra we drop the subscript. We fix a Cartan involution $\theta$ of $\mathfrak{g}_0$ and extend to $\mathfrak{g}$, let $\Theta$ be the corresponding involution of $G$. A maximal compact subgroup of $G$ is $K$, the fixed space of $\Theta$. The Lie algebra $\mathfrak{g}_0$ decomposes as $\mathfrak{k}_0 \oplus \mathfrak{p}_0$. The subspace  $\mathfrak{p}_0$ is the $-1$ eigenspace of $\theta$, the subalgebra $\mathfrak{k}_0$ is the $+1$ eigenspace of $\theta$ and the Lie algebra of $K$. Similarly, $\mathfrak{g} = \mathfrak{k} \oplus \mathfrak{p}$.  Let $\mathfrak{a}_0$ be a maximal commutative Lie subalgebra of $\mathfrak{p}_0$. 
Let $M$ be the centralizer of $\mathfrak{a}_0$ in $K$ under the adjoint action. We have $\Lie(M) =\mathfrak{m}_0$. 

\begin{definition} For $G$ equal to $O(p,q)$ or  $Sp_{2n}(\mathbb{R})$ write $V_0$ for the defining matrix module. That is $\rho : G \to GL(V_0)$ is the injection defining $G$ as a linear group. Write $V = V_0\otimes_\mathbb{R} \mathbb{C}$ for the complexification of $V_0$. 
\end{definition} 
If $G=Sp_{2n}(\mathbb{R})$ then $V=\mathbb{C}^{2n}$ and if $G =  O(p,q)$ then $V = \mathbb{C}^{2n+1}$. 
When $G = Sp_{2n}(\mathbb{R})$, let $e_1,...,e_{2n}$  be the standard matrix basis of $V$, then define a new basis $f_i = e_i + e_{n+i}$ for $i = 1,..,n$ and $f_i' = e_i - e_{n+i}$  for $i = 1,..,n$. We also label $f_i$ by $f_i^1$ and $f_i'$ by $f_i^{-1}$. When $G = O(p,q)$ then $V$ has basis $e_1,...,e_{2n+1}$, we let $f_i = e_{p-i+1} + e_{p+i}$ and $f_i' = e_{p-i+1} - e_{p+i}$.

 Following \cite[Section 1.1]{L89}, let $\{R, X, \hat{R}, \hat{X},\Delta\}$ be root datum where $R$ is the set of roots, $\hat{R}$ is the set of coroots and $X$ and $\hat{X}$ are free groups that contain $R$ and $\hat{R}$ respectively. There is a perfect pairing $\langle, \rangle$ between $X$ and $\hat{X}$ which defines a pairing between $R$ and $\hat{R}$. The simple roots $\Delta$ are a subset of $R$. Let $\mathfrak{t}$ equal the complexification of $X$, and similarly $\hat{\mathfrak{t}}$ is the complexification of $\hat{X}$.

The Lie algebra $\mathfrak{g}$ decomposes  as $\mathfrak{k} \oplus \mathfrak{p}$. Let $\mathfrak{a}$ be a maximal abelian Lie subalgebra of $\mathfrak{p}$. The restricted roots $\Sigma$ of $\mathfrak{g}$ are given by the eigenvalues of $\mathfrak{a}$ acting on $\mathfrak{g}$. The nilpotent Lie subalgebra $\mathfrak{n}$ is  the sum of positive root spaces of the restricted roots of $\mathfrak{a}$.    
\begin{definition}{\cite[Proposition 6.46]{K96}, \cite{I49}} The Iwasawa decomposition of the complex vector space $\mathfrak{g}$ is 
$$\mathfrak{g} = \mathfrak{k}\oplus \mathfrak{a}\oplus \mathfrak{n}.$$
The Iwasawa decomposition of $G$ \cite[Theorem 6.46]{K96} is 
$$G = KAN.$$  
\end{definition}

Let $M$ be the centralizer of $\mathfrak{a}$ in $K$ and denote by $N_K(\mathfrak{a})$ the normalizer of $\mathfrak{a}$ in $K$.  Let $\mathfrak{m}_0$ be the Lie algebra of $M$ with complexification $\mathfrak{m}$. The Weyl group associated to $G$ is the group
$$W_G =  N_K(\mathfrak{a})/M.$$

\begin{example} \label{symplecticexample2}

For $G= Sp_{2n}(\mathbb{R})$, a maximal abelian  subalgebra $\mathfrak{a}$ of $\mathfrak{p}$ is
$$\mathfrak{a} =\left \{\begin{bmatrix} 0 & B \\B & 0\end{bmatrix}:  B \text{ is diagonal}, B \in \mathfrak{gl}_n(\mathbb{C})\right\}.$$
The algebra $\mathfrak{a}$ has dimension $n$, this is the real rank of $Sp_{2n}(\mathbb{R})$. 
Let $E_{i,j}$ be the matrix with $1$ in the $(i,j)$ position and zero elsewhere. Let $k \in \{0,...,n\}$ The subspace $\mathfrak{a}_k$ is the span of $E_{i,n+i} + E_{n+i,i}$ for $i = 1,...,k$. 
 The subspace $\bar{\mathfrak{a}}_{n-k}\subset \mathfrak{a} $ is the span of the vectors $E_{k+i,n+k+i} + E_{n+k+i,k+i}$ for $i = 1,...,n-k$. Note that 
$$\mathfrak{a} = \mathfrak{a}_k \oplus \bar{\mathfrak{a}}_{n-k}.$$ 
For $G = SO(p,q)$ there is a similar decomposition of $\mathfrak{a}$ into  $k$ dimensional and $q-k$ dimensional subspaces, which we label by $\mathfrak{a}_k$ and $\bar{\mathfrak{a}}_{q-k}$. 
\end{example}

\begin{definition}\label{casimir} Given a finite dimensional complex Lie algebra $\mathfrak{g}$ with basis $B$ and dual basis $B^*$ with respect to the Killing form, we define the Casimir element in the enveloping algebra $U(\mathfrak{g})$ to be 
$$C^\mathfrak{g} =\sum_{b\in B} bb^* \in U(\mathfrak{g}).$$
For a subalgebra $\mathfrak{h} \subset \mathfrak{g}$ we denote the Casimir element of $\mathfrak{h}$ in $\mathfrak{g}$ by $C^\mathfrak{h} = \sum_{b \in B \cap \mathfrak{h}} b b^*$ where the basis $B$ is taken such that $B \cap \mathfrak{h}$ is a basis of $\mathfrak{h}$. 
\end{definition}

\end{section}

\begin{section}{Brauer Algebras}\label{brauer}
For a given $\mathfrak{g}$-module $X$ and the matrix module $V$, the endomorphism ring $\End_\mathbb{C}(X\otimes V ^{\otimes k})$ has been thoroughly studied. Most attention (\cite{We88,CDM06,DDH08,BS12,ES14}) has been on understanding the subalgebra commuting with $G$: $$\End_G(X \otimes V^{\otimes k}).$$
In the case of $\mathfrak{g} = \mathfrak{gl}_n$, $\End_{\mathfrak{gl}_n}(X \otimes V^{\otimes k})$ has a map from the graded Hecke algebra associated to the symmetric group \cite{AS98}. However , with $\mathfrak{g} = \mathfrak{so}_{2n+1},$  the relevant algebra is the VW-algebra with parameter $n$. With $\mathfrak{g}=\mathfrak{sp}_{2n}$, the corresponding algebra is the VW-algebra with parameter  $-n$.

In this section, we define the type $B/C$ Brauer algebra as an extension of the VW-algebra. We endow it with a natural action on $X \otimes V ^{\otimes k}$ and prove that it commutes with the action of $K$.  The degenerate BMW algebra is a quotient of the VW-algebra. We choose not to use the BMW algebra \cite{DRV12} as we are fundamentally interested in resulting graded Hecke algebra modules, the quotients to Hecke algebras defined in Section \ref{quotientsofbrauer} annihilate the difference between the VW-algebra and the degenerated BMW algebra.

\begin{definition} \cite{B37} The rank $k$ Brauer algebra $B_k[m]$, with parameter $m\in \mathbb{C}$, is the associative $\mathbb{C}$-algebra generated by elements $t_{i,i+1}$ and $e_{i,i+1}$ for $i =1,...,k-1$, subject to the conditions:
$$\text{ the subalgebra generated by } t_{i,i+1} \text{ is isomorphic to } \mathbb{C}[S_n],$$
$$e_{i,i+1}^2 = me_{i,i+1}, $$
$$t_{i,i+1} e_{i,i+1} = e_{i,i+1} t_{i,i+1} = e_{i,i+1},$$
$$t_{i,i+1}t_{i+1,i+2} e_{i,i+1}t_{i+1,i+2}t_{i,i+1} = e_{i+1,i+2},$$
$$[t_{i,i+1},e_{j,j+1}] = 0 \text{ for } j \neq i,i+1.$$
\end{definition}

\begin{definition}\label{affinebraueralgebradef} Let $U$ be a vector space with basis $z_1,...,z_k$.   The rank $k$ VW-algebra $\mathfrak{B}_k[m]$, with parameter $m\in \mathbb{C}$ is as a vector space equal to $$\mathfrak{B}_k[m] \cong B_k[m] \otimes S(U).$$ The multiplication satisfies the following conditions:
$$t_{i,i+1}z_i - z_{i+1}t_{i,i+1} = 1 + e_{i,i+1} ,$$
$$[t_{i,i+1}, z_j ] = 0, j \neq i,i+1,$$
$$e_{i,i+1}(z_i+z_{i+1}) = 0 =(z_i+z_{i+1}) e_{i,i+1},$$
$$[e_{i,i+1}, z_j ] = 0, j \neq i,i+1,$$
$$[z_i,z_j] =0,$$
$$e_{12}z_1^le_{12} = W_l e_{12} \text{ for  constants } w_l \in \mathbb{C},$$ 
 the subalgebra generated by $t_{i,i+1}$, $e_{i,i+1}$ is isomorphic to $B_k[m]$.

\end{definition}

 Let us consider $X$ and $V$ as $U(\mathfrak{g})$-modules then $X \otimes V^{\otimes k}$ has a $U(\mathfrak{g})^{\otimes k+1}$-module structure. 
We define operators that form the action of the Brauer algebra. 

\begin{definition} Given the action of $U(\mathfrak{g})^{\otimes k+1}$ on $X \otimes V^{\otimes k}$ we write $(g)_i$ for the action of $g$ on the $i+1^{st}$ tensor in $U(\mathfrak{g})^{\otimes k+1}$,
$$(g)_i = \overbrace{ id \otimes ...\otimes id }^{i \text{ times}} \otimes g \otimes\overbrace{id\otimes ...\otimes id}^{k-i \text{ times}}.$$
\end{definition}
By construction we start counting from zero. Hence $(g)_0 = g \otimes id\otimes...\otimes id \in U(g)^{\otimes k+1}$.

\begin{definition}\label{Omega} Fix a basis $B$ of $\mathfrak{g}$  such that $B = (B \cap \mathfrak{k}) \bigcup (B \cap \mathfrak{p})$. Let $B^*$ be the dual basis with respect to the Killing form of $\mathfrak{g}$. For $0\leq i < j \leq k,$ define $\Omega_{ij}$ to be the operator
$$\Omega_{ij} = \sum_{b\in B}(b)_i\otimes (b^*)_j \in U(\mathfrak{g})^{\otimes k+1}.$$
Similarly we define $\Omega^\mathfrak{k}_{ij}$ and $\Omega^\mathfrak{p}_{ij}$ as
$$\Omega^\mathfrak{k}_{ij} = \sum_{b\in B\cap \mathfrak{k}}(b)_i\otimes (b^*)_j \in U(\mathfrak{g})^{\otimes k+1},$$
$$\Omega^\mathfrak{p}_{ij} = \sum_{b\in B\cap \mathfrak{p}}(b)_i\otimes (b^*)_j \in U(\mathfrak{g})^{\otimes k+1}.$$

\end{definition} 

\begin{lemma} The operators $\Omega_{ij}$,$\Omega^\mathfrak{k}_{ij}$ and $\Omega^\mathfrak{p}_{ij}$ are independent of the choice of basis of $\mathfrak{g}$, $\mathfrak{k}$ and $\mathfrak{p}$ respectively.\end{lemma}

\begin{proof}  It is sufficient to prove the statement for $\Omega_{12} \in U(\mathfrak{g})^2$. Let  $C^\mathfrak{g} = \sum_{b \in B} b b^* \in U(\mathfrak{g})$ be the Casimir element and $\Delta:U(\mathfrak{g}) \to U(\mathfrak{g}) \times U(\mathfrak{g})$ denote comultiplication. The operator $\Omega_{12}$ can be written as:
$$\Omega_{12} = \Delta (C^\mathfrak{g} ) - C^\mathfrak{g}  \otimes 1 - 1 \otimes C^\mathfrak{g} .$$
 The Casimir element $C^\mathfrak{g}$ is independent of the choice of basis therefore $\Omega_{12}$ is also independent. Similarly replacing $\mathfrak{g}$ with the Lie subalgebra $\mathfrak{k}$, $\Omega^\mathfrak{k}_{12}$ is independent of choice of basis. Finally $\Omega^\mathfrak{p}_{12}$ is independent as it is the difference of the other two,
$$\Omega_{ij} - \Omega^\mathfrak{k}_{ij} =\Omega^\mathfrak{p}_{ij}.$$
\end{proof}

Let  the symmetric group on $k$ elements $S_k$ act on $X \otimes V^{\otimes k}$ by permuting the factors of $V$. Explicitly the simple reflection $s_{i,i+1}$ acts by 
$$s_{i,i+1} (x_0\otimes v_1\otimes..\otimes v_i \otimes v_{i+1} \otimes ....\otimes v_k) =x_0\otimes v_1\otimes..\otimes v_{i+1} \otimes v_i \otimes ....\otimes v_k.$$

\begin{lemma}\label{decompofVtensorV}  If $\mathfrak{g} = \mathfrak{sp}_{2n}$ or $\mathfrak{so}_{2n+1}$ then $V \otimes V$ decomposes as 
$$\Lambda^2 V \oplus S^2V/1 \oplus 1 \text{ for } \mathfrak{so}_{2n+1},$$
$$\Lambda^2 V / 1 \oplus S^2 V \oplus 1 \text{ for } \mathfrak{sp}_{2n}.$$
Here $1$ denotes the trivial module of $\mathfrak{g}$.
\end{lemma}
Let $\pr_1$ be the projection of $V\otimes V$ onto the trivial submodule $1$ in the decomposition above. Let $\pr_{i,i+1}$ be the projection onto the trivial submodule of $V_i \otimes V_{i+1}$.

\begin{lemma}{\cite[Theorem 2.2]{DRV12}}\label{braueraction}   Let $G$ be $O(p,q)$ or $Sp_{2n}(\mathbb{R})$. Let $X$ be a complex $G$-representation and $V$ the defining matrix module of $G$.  Then there exists $m\in \mathbb{N}$ such that there is a homomorphism $\pi:\mathfrak{B}_k[m] \to \End(X \otimes V^{\otimes k})$:
$$\pi( z_i) = \sum_{j<i}\Omega_{ji},$$
$$\pi(t_{i,i+1}) = s_{i,i+1},$$
$$\pi(e_{i,i+1}) = id \otimes ...\otimes id \otimes m\pr_{i,i+1}\otimes id \otimes....\otimes id.$$
For $G = Sp_{2n}(\mathbb{R})$ the parameter is $m =-n$ and if $G = O(p,q)$ then $m = \lfloor \frac{p+q}{2}\rfloor$.
\end{lemma}

\begin{theorem}  For $G = O(p,q)$ or $Sp_{2n}(\mathbb{R})$, the VW-algebra with the action on $X\otimes V ^{\otimes k}$ defined in Lemma \ref{braueraction} commutes with the action of $U(\mathfrak{g})$ on $X \otimes V^{\otimes k}$.

\end{theorem}\begin{lemma}{\cite[Lemma 2.3.1]{CT11}} \label{omegatranspositions} 
Let $0<i<j\leq k $ and $G = O(p,q)$ or $Sp_{2n}(\mathbb{R})$.  As operators on $X \otimes V^{\otimes k}$ 
$$\Omega_{ij} = s_{ij} + m\pr_{i,i+1},\text{ for } 1 \leq i<j \leq k$$
where $m = \lfloor \frac{p+q}{2}\rfloor$ or $-n$ respectively.
\end{lemma}
\begin{proof}
One only needs to consider the operator $\Omega_{12}$ on $V \otimes V$. By Lemma \ref{decompofVtensorV} $V\otimes V$ decomposes as 
$$\Lambda^2 V \oplus S^2V/1 \oplus 1 \text{ for } \mathfrak{g}_\mathbb{C}=  \mathfrak{so}_{2n+1}(\mathbb{R}),$$
$$\Lambda^2 V / 1 \oplus S^2 V \oplus 1 \text{ for } \mathfrak{g}_\mathbb{C} = \mathfrak{sp}_{2n}(\mathbb{C}).$$
On $V\otimes V$ $s_{12} = pr_{S^2V} -pr_{\Lambda^2V}$. Then using the fact that $\Omega_{12} = \Delta(C) - C \otimes 1 - 1 \otimes C$ we find the operators
$$\Omega_{12} \text{ and } s_{12} + m e_{12},$$
act by the same scalars on the irreducible decomposition of $V \otimes V$. \end{proof}

For $G=GL_n$ the commutator $\End_{GL_n}(X \otimes V^{\otimes k})$ contains the same type Weyl group, the symmetric group (\cite{AS98}). One might expect that in type $B$ and $C$ this may be the case too. However $\End_{Sp_{2n}(\mathbb{R})}(X\otimes V^{\otimes k})$, $\End_{O(p,q)}(X\otimes V^{\otimes k})$ and the VW-algebra, do not contain a copy of the hyperoctahedral group. We look to establish a theory that has this type symmetry reflected in the commutator. 

We introduce the type $B/C$ VW-algebra which acts on $X\otimes V^{\otimes k}$ and commutes with the action of $K$ for $G = Sp_{2n}(\mathbb{R})$ or $O(p,q)$. Crucially the type $B/C$ VW-algebra contains the Weyl group of Type B/C, $W(B_k)$. Recall the hyperoctahedral group is generated by simple reflections $s_{\epsilon_i - \epsilon_{i+1}}$ and $s_{\epsilon_k}$.

\begin{definition}\label{BCbrauerdef}
The type $B/C$  VW-algebra $\mathfrak{B}^\theta_k[m_0,m_1]$  is generated by the VW-algebra $\mathfrak{B}_k[m_0]$ and reflections $\theta_j$, for $j = 1,...,k$, such that the subalgebra generated by $t_{i,i+1}$, for $i =1,..,k-1$ and $\theta_j$ is isomorphic to the group algebra of the $k^{th}$ hyperoctahedral group $\mathbb{C}[W(B_k)]$ and the following relations hold;

$$[e_{i,i+1},\theta_j]=0 \text{ for all j},$$
$$e_{i,i+1} \theta_i\theta_{i+1}= e_{i,i+1} =\theta_i \theta_{i+1}e_{i,i+1}  \text{ for } i = 1,...,k-1,$$
$$[\theta_n, x_j] = 0 \text{ for } j \neq k.$$
$$e_{i,i+1}\theta_i e_{i,i+1} = m_1 e_{i,i+1} \text{ for } i =1,...,k-1,$$
\end{definition} 

The Lie algebra $\mathfrak{g}$ decomposes as eigenspaces of a Cartan involution $\theta$ that is $\mathfrak{g} = \mathfrak{k} \oplus \mathfrak{p}$. For $G = O(p,q)$ or $Sp_{2n}(\mathbb{R})$ there is a semisimple involutive $\xi \in \mathfrak{g}$ such that $\theta$ is equal to conjugation by $\xi$.
\begin{remark} The subalgebra of $\mathfrak{B}^\theta_k[m_0,m_1]$ generated by $e_{i,i+1}$, $t_i$ and $\theta_i$ is equal to the cyclotomic Brauer $Br_{k,2}[m_0,m_1]$, see \cite{H01,BCV12,C20} for the definition of the cyclotomic Brauer algebra, it's representation theory and how it acts on $\End_k(V^{\otimes k})$. \end{remark}
\begin{lemma}\label{actionofB} The type $B/C$ VW-algebra $\mathfrak{B}^\theta_k[m]$ acts on $X \otimes V^{\otimes k}$. This action is defined by extending the action $\pi$ of the VW-algebra to the extra generators $\theta_i$. The generators $\theta_i$ act by $(\xi)_i \in U(\mathfrak{g})^{\otimes k+1}$. 
Extend $\pi$ to  $\mathfrak{B}^\theta_k[m]$ by $\pi(\theta_i) = (\xi)_i \in U(\mathfrak{g})^{\otimes k+1}\subset \End(X\otimes V^{\otimes k})$. That is
$$\pi: \mathfrak{B}^\theta_k[m] \longrightarrow \End_K(X \otimes V^{\otimes k}),$$

$$\pi (\theta_i) = (\xi)_i.$$
\end{lemma}
Explicitly, $(\xi)_i = \overbrace{id \otimes ...\otimes id }^i \otimes \xi \otimes  \overbrace{id \otimes ...\otimes id}^{k-i} \in U(\mathfrak{g})^{\otimes k+1}.$
The constants $(m_0,m_1)$ equal $(\lfloor \frac{p+q}{2}\rfloor,p-q)$ when $G= O(p,q)$ and $(m_0,m_1) = (-n,0)$ if $G= Sp_{2n}(\mathbb{R})$.
\begin{proof} Since we know that the VW-algebra  $\mathfrak{B}_k[m]$ acts on $X\otimes V^{\otimes k}$ and that the cyclotomic Brauer algebra $Br_{k,2}[m_0,m_1]$ acts on $\End_K(V^{\otimes k})$ we only need to check the action of $\theta_j$, and $z_i$ and the relations involving them in Definition \ref{BCbrauerdef}. This equates to checking $[z_i, \theta_k] = 0 $ for all $i < n$.

If $i \neq j$, then $(g)_i$ and $(h)_j$ commute  in $U(\mathfrak{g})^{k+1}$. 
Definition \ref{braueraction} states $\pi(z_i) = \sum_{j<i} \Omega_{ji}$, hence:
$$[z_i, (\xi)_k ] = \sum_{j<i} [\Omega_{ji}, (\xi)_k ] = 0 \text{ for } k <n.$$
\end{proof}

\begin{theorem} Let $G = O(p,q) $ or $Sp_{2n}(\mathbb{R})$ and $X$  a Harish-Chandra module. The type $B/C$ Brauer algebra $\mathfrak{B}^\theta_k[m]$ acts on $X \otimes V^{\otimes k}$ and commutes with the action of $K$ on $X\otimes V ^{\otimes k}$.
\end{theorem}

\begin{proof} The action of $\mathfrak{B}_k[m]$ commutes with $\mathfrak{g}$ and by restriction with $K$. The algebra $\mathfrak{B}_k^\theta[m_0,m_1] = \langle \mathfrak{B}_k[m] , \theta_j : \text{ for } j = 1,...,k\rangle$. Therefore, to verify that $\mathfrak{B}^\theta_k[m]$ commutes with the action of $K$, one only needs to check that $\pi(\theta_j) =(\xi)_j$ commutes with the action of $K$. Conjugation by $\xi$ is the Cartan involution: $\xi^{-1} K \xi = \Theta(K)$. By definition, $\Theta$ is the identity on $K$. Hence $\xi k - k \xi = 0$ for $k\in \mathfrak{k}$. Therefore:
$$[(\xi)_i, k] = \sum_{j=0}^{k+1} (\xi)_i -(k)_j  (\xi)_i (k)_j  = 0.$$

Hence the action of $(\xi)_i$ and $K$ commute. 
\end{proof}

\end{section}

\begin{section}{Quotients of the type $B/C$ Brauer algebra $\mathfrak{B}_k^\theta$}\label{quotientsofbrauer}

In Section \ref{functorssection} we introduce functors, defined in \cite{CT11}, from the category $\mathcal{HC}(G)$-mod to the category of $\mathfrak{B}^\theta_k[m]$  modules. However, we are aiming at graded Hecke algebra modules. In this section, we look at particular ideals and quotients of $\mathfrak{B}^\theta_k[m]$ which are isomorphic to graded Hecke algebras. This will set up Section \ref{restrictiontoprincipal} in which we focus on principal series modules and show that via the quotients defined in this section, the functors defined in Section \ref{functorssection} descend to take principal series modules to graded Hecke algebra modules. 

Recall that $W(R)$ denotes the Weyl group associated to a root datum $(R,X,\hat{R}, \hat{X},\Delta)$ and $\langle, \rangle: X \times \hat{X} \to \mathbb{C}$ is the pairing between dual spaces. Define the $\mathbb{C}$-spaces $\mathfrak{t} = X \otimes_\mathbb{Z} \mathbb{C}$, $\mathfrak{t}^* = \hat{X} \otimes_\mathbb{Z} \mathbb{C} .$

\begin{definition}\label{Heckedef}\cite{L89} The graded Hecke algebra $\mathbb{H}^R(\textbf{c})$ associated to the root system $(R,X ,\hat{R},\hat{X},\Delta)$ and parameter function $\mathbf{c}$ from $\Delta$ to $\mathbb{C}$, is as a vector space
$$\mathbb{H}^R(\textbf{c}) \cong S(\mathfrak{t}) \otimes \mathbb{C}[W(R)],$$
such that as an algebra $S(\mathfrak{t})$ and $\mathbb{C}[W(R)]$ are subalgebras and the following cross relations hold,
$$s_\alpha \epsilon - s_\alpha(\epsilon)s_\alpha = \mathbf{c}(\alpha)\langle \alpha , \hat{\epsilon} \rangle, \text{ for } \epsilon \in \mathfrak{t} \text{ and } \alpha \in \Delta.$$
\end{definition} 

If the parameter function $\mathbf{c}:\Delta \to \mathbb{C}$ is taken to uniformly be 1, then in this case the graded Hecke algebra is entirely defined by the root system $(W,R,\Delta)$. For a Hecke algebra determined by the uniform parameter we denote it by $\mathbb{H}^{R_k}$ where $R_k$ is the root system. For example $\mathbb{H}^{D_k}$ denotes the graded Hecke algebra associated to the root system $D_k$ with the parameter function $\textbf{c}: \delta \to \mathbb{C}$ such that $\textbf{c}(\alpha) \equiv1.$

We fix the set of simple reflections of the hyperoctahedral group $W(B_k)$ to be $\{s_{i,i+1},\theta_k: i = 1,...,k-1\}$. We also associate to the hyperoctahedral group a $k$ dimensional vector space $\mathfrak{t}$ with basis $\epsilon_1,...,\epsilon_k$ and subset $\Delta = \{\epsilon_i-\epsilon_{i+1} \text{ and } \epsilon_k: i = 1,...,k-1\}$. Then for $c\in \mathbb{C}$ we define the parameter $c: \Delta\to \mathbb{C}$ as 
$$c(\alpha) = \begin{cases} 1 & \text{ if } \alpha = \epsilon_i-\epsilon_{i+1},\\c &\text{ if } \alpha = \epsilon_k.\end{cases}$$

We denote the graded Hecke algebra associated  to the Weyl group $W(B_k)$ with the parameter $c$ as $\mathbb{H}_k(c).$

\begin{lemma}\label{hecketypes} The graded Hecke algebra of type $B_k$ (resp. type $C_k$) is isomorphic to $\mathbb{H}_k(1)$ (resp. $\mathbb{H}_k(\frac{1}{2}))$) and the algebra $\mathbb{H}_k(0)$ is isomorphic to an extension of the Hecke algebra of type $D_k$,
$$\mathbb{H}_k(0) \cong \mathbb{H}^{D_k} \rtimes \mathbb{Z}_2.$$
\end{lemma}

\begin{proof} The isomorphism of $\mathbb{H}_k(1)$ and the graded Hecke algebra $\mathbb{H}^{B_k}$ is apparent from the definitions. 
The space $\mathfrak{t}_{D_k}$ is equal to the space $\mathfrak{t}$ in $\mathbb{H}_k(0)$. The Weyl group $W(D_k)$ is naturally a subgroup of $W(B_k)$. The generator $t \in \mathbb{Z}_2$ acts on $\mathbb{H}^{D_k}$ by interchanging roots $\epsilon_{k-1}- \epsilon_k$ and $\epsilon_{k-1} + \epsilon_k$ and acts by conjugation by $s_{\epsilon_k} \in W(B_k)$ on $W(D_k) \subset W(B_k)$. 

\end{proof}

We define two ideals in the type $B/C$ VW-algebra $\mathfrak{B}^\theta_k[m]$.  We then show that the quotient of $\mathfrak{B}^\theta_k[m]$ by these ideals is isomorphic to a graded Hecke algebra. 

\begin{definition} Let $I_e$ be the two sided ideal in $\mathfrak{B}^\theta_k[m]$  generated by the idempotents,
$$\{e_{i,i+1}: \text{ for } i =1,...,k-1\}.$$
Let $c \in \mathbb{C}$ and $r \in \mathbb{Z}$, define $I_c^r$ to be the two sided ideal,
$$I_c^r = \langle\theta_k z_k +z_k \theta_k - 2c +2r\theta_k\rangle .$$
\end{definition}

The ideal $I_e$ can be generated by any idempotent since they are all in the same $S_k$ conjugation orbit.   By using $c \in \mathbb{C}$ we have abused notation; however the two occurrences of $c$ will correspond to the same constant.

\begin{lemma}\label{quotienttohecke} The quotient of the algebra $\mathfrak{B}^\theta_k[m]$ by the ideal generated by $I_e$ and $I_c^r$ is isomorphic to the graded Hecke algebra 
$$\mathfrak{B}_k^\theta[m_0,m_1]/\langle I_e,I_c^r\rangle \cong \mathbb{H}_k(c).$$
\end{lemma} 
\begin{proof} 

Consider the presentation in Definition \ref{BCbrauerdef} with generators
$$z_i,\theta_j, t_{i,i+1},e_{i,i+1}$$
and relations 
$$\theta_j^2 = 1, s_{i,i+1}^2 = 1, (s_{i,i+1}s_{i+1.i+2})^3 = 1, (s_{k-1,k}\theta_k)^4 = 1,$$
$$t_{i,i+1} z_i - x_{i+1}t_{i,i+1} = 1 + e_{i,i+1} ,$$
$$[t_{i,i+1}, z_j ] = 0, j \neq i,i+1,$$
$$e_{i,i+1}(z_i+z_{i+1}) = 0 =(z_i+z_{i+1}) e_{i,i+1},$$
$$[e_{i,i+1}, z_j ] = 0, j \neq i,i+1,$$
$$[z_i,z_j] =0,$$
$$[e_{i,i+1},\theta_j]=0 \text{ for all j},$$
$$e_{i,i+1} \theta_i\theta_{i+1}= e_{i,i+1} =\theta_i \theta_{i+1}e_{i,i+1}  \text{ for } i = 1,...,k-1,$$
$$[\theta_n, z_j] = 0 \text{ for } j \neq k,$$
$$e_{12}z_1^l e_{12} = w_l e_{12}.$$
Under the quotient by $I_e$ and $I_c^r$ the generators $e_{i,i+1}$ and the relations $e_{i,i+1} =0$ cancel out. Furthermore we add another relation: $z_k \theta_k + \theta_k z_k -2c + 2r \theta_k$. 
Hence the presentation has generators 
$$z_i,\theta_j, t_{i,i+1}$$
with relations 
$$\theta_j^2 = 1, s_{i,i+1}^2 = 1, (s_{i,i+1}s_{i+1.i+2})^3 = 1, (s_{k-1,k}\theta_k)^4 = 1,$$
$$t_{i,i+1} z_i - z_{i+1}t_{i,i+1} = 1,$$
$$[t_{i,i+1}, z_j ] = 0, j \neq i,i+1,$$
$$[z_i,z_j] =0,$$

$$[\theta_n, z_j] = 0 \text{ for } j \neq k,$$
$$z_k \theta_k + \theta_k z_k -2c + 2r \theta_k.$$
This is a presentation of the Hecke algebra $\mathbb{H}_k(c)$; it is the modification of the presentation in Definition \ref{Heckedef} by $\epsilon_i \mapsto z_i +r$. Since we have shown that the presentation of  $\mathfrak{B}_k^\theta[m_0,m_1] /\langle I_e I_c^r\rangle$ is identical to the presentation of $\mathbb{H}_k(c)$ then these algebras are isomorphic. 

\end{proof} 
\begin{remark} We could have chosen to quotient by the ideal generated by $\theta_k z_k+ z_k \theta_k -c$ without the $2r\theta_k$ part. This quotient would also be isomorphic to $\mathbb{H}_k(c)$ with $\epsilon_i$ mapping to $z_i$. However, we need the modification of the affine parts by the scalar $r$ to enable our results regarding images of principal series modules descending to Hecke algebra modules. One can think of this modification by $r$ as an analogue of the $\rho$ shift.\end{remark}

\end{section} 

\begin{section}{Functors from $\mathcal{HC}(G)$-mod to $\mathfrak{B}^\theta_k$-mod }\label{functorssection}

In this section, we introduce functors, defined in \cite{CT11}. We show these functors take Harish-Chandra modules to modules of the  $\mathfrak{B}^\theta_k$ algebra. 

\begin{definition}\label{functorsdef} \cite[(2.8)]{CT11}  Let $n$ be the real rank of $G$. If $G= Sp_{2n}(\mathbb{R})$ the real rank is $n.$ If $G = O(p,q)$ then $n = q = \min(p,q)$. Let $\mu$ be an irreducible $K$-module, fix an integer $k \leq n$. The space $V$ is the matrix module of $G$. We define the functor $F_{\mu, k}$  to be:
$$F_{\mu,k}: \mathcal{HC}(G)\text{-mod} \longrightarrow \mathfrak{B}^\theta_k\text{-mod}$$
$$X \mapsto \Hom_K(\mu, X\otimes V^{\otimes k}),$$
and on morphisms $f: X \to Y$ and $g \in \Hom_k(\mu, X \otimes V^{\otimes k})$,
$$F_{\mu ,k}f (g): \mu  \to Y \otimes V^{\otimes k},$$
$$F_{\mu ,k}f (g)( \mu) = f\otimes id ^{\otimes k}  g(\mu).$$
\end{definition}

\begin{remark} Lemma \ref{actionofB} gives an action of  $\mathfrak{B}^\theta_k$ on $X\otimes V^{\otimes k}$. Since this action  commutes with the action of $K$ then $\mathfrak{B}^\theta_k$ naturally acts on $\Hom_K(\mu, X\otimes V ^{\otimes k})$ from the inherited action on $X \otimes V^{\otimes k}$.  
\end{remark} 

\begin{lemma}\label{exactness} For any irreducible $K$-module $\mu$ and $k \leq n$, the functor $F_{\mu ,k}$ defined in Definition \ref{functorsdef} is exact. \end{lemma}

\begin{proof} Tensoring with a finite dimensional module is exact. The module $V^{\otimes k}$ is finite dimensional hence the functor taking $X$ to $X \otimes V^{\otimes k}$ is exact. Furthermore, $\mu$ is an irreducible $K$-module. Therefore the functor which takes $Y$ to $\Hom_K(\mu,Y)$ is exact. The functor $F_{\mu, k}$ is the composition of these two exact functors, hence the result follows. \end{proof}

\end{section}

\begin{section}{Restricting functors to principal series modules}\label{restrictiontoprincipal}

The functors (Definition \ref{functorsdef}) take any Harish-Chandra module to a $\mathfrak{B}^\theta_k$-module.  In this section, given a principal series module we give a basis for the image of the functors $F_{\mu,k}$ and $F_{\underline{\mu},n-k}$ for particular characters $\mu,\underline{\mu}$ depending on the principal series modules. 

Let $G = Sp_{2n}(\mathbb{R})$ then  $K \cong U(n)$, $M \cong (\mathbb{Z}_2)^n$. The Cartan involution $\theta$ is equal to conjugation by the matrix
$$\xi = \begin{bmatrix}  0& i \Id_n\\-i \Id_n&0\end{bmatrix}.$$
The subspace  $\mathfrak{a}$ has dimension $n$ with basis $\varepsilon_i$ and corresponds to the subgroup $A$ under the exponential map. We label a character of $\mathfrak{a}$ by $\nu \in \mathfrak{a}^*$ and characters of $A$ by $e^\nu$. The matrix module $V \cong \mathbb{C}^{2n}$ has two bases:  $\{e_1,...,e_{2n}\}$ and $\{f_1^1,..,f_n^1,f_1^{-1},...,f_n^{-1}\}$, where $f_i^\eta = e_i + \eta e_{n+i}$. 

Recall that the Iwasawa decomposition of $G$ is
$$G =KAN,$$
also, that $M$ is the centraliser of $\mathfrak{a}_0$ in $K$, which is isomorphic to $\mathbb{Z}_2^{n}$. The character $\delta^k$ is defined to be the character of $M$ which takes the first $k$ generators of $\mathbb{Z}_2^n$ to $-1$ and the last $n-k$ to $1$. We write $1$ for the trivial character of $N$. 

If $G = O(p,q)$ then $K \cong O(p) \times O(q)$, $M = O(p-q) \times O(1)^{q}$ embedded into $O(p,q)$ as the block matrix $$ ( O(p-q), x_1,x_2,...,x_q,x_q,..,x_1)$$
where $x_i \in O(1)$. We denote characters of $M$, $\delta^k_{\triv}$ and $\delta^k_{\det}$ to be
$$\delta^k_{\triv} = \triv \otimes (\sgn^k) \otimes \triv^{q-k} \text{ on } O(p-q) \otimes O(1)^q,$$
$$\delta^k_{\det} = \det \otimes (\sgn^k) \otimes \triv^{q-k}\text{ on } O(p-q) \otimes O(1)^q.$$
The Cartan involution $\theta$ is equal to conjugation by the matrix
$$\xi = \begin{bmatrix}  \Id_p&0\\0&-\Id_q\end{bmatrix}.$$
\begin{definition}\cite{V15} \label{principalseriesdef} Let $G = KAN$ (resp. $\mathfrak{g}_0 = \mathfrak{k}_0 \oplus \mathfrak{a}_0 \oplus \mathfrak{n}_0$) be the Iwasawa decomposition of $G$ (resp. $\mathfrak{g}_0)$ and let $M$ be the centraliser of $\mathfrak{a}_0$ in $K$. Given a character $e^\nu$ of $A$ and the character $\delta$ of $M$ we define the minimal principal series representation; 
$$X_{\delta}^\nu = \Ind_{MAN}^G (\delta \otimes e^\nu \otimes 1).$$
\end{definition}
In the non-split case principal series representations may be induced from irreducible representations of $M$ which are not one dimensional. In this chapter we will only study principal series modules that are induced from a character of $M$.
We write $\mathbbm{1}_{\delta}^\nu$ for the vector spanning the representation space of the character $\delta \otimes e^\nu \otimes 1$. Hence $$X_{\delta}^\nu = \Ind_{MAN}^G \mathbbm{1}_{\delta}^\nu.$$

For $G= Sp_{2n}(\mathbb{R})$, we can calculate the dimension of $F_{\triv, k}(X_{\delta^k}^\nu$) and $F_{\Det, n-k} (X_{\delta^k}^\nu)$. Note that if we want to describe the trivial isotypic component we must take $F_{\triv,k}$ and if we wish to look at the $\det$ isotypic component then we must take the functor $F_{\det , n-k}$. 

For $G = O(p,q)$, we can calculate the dimension of $F_{\triv \otimes \sgn, k}$ and$F_{\triv \otimes \triv,q-k}$. Similarly for $X_{\delta^k_{\det}}^\nu$, we take the functors $F_{\sgn \otimes \triv, k}$ and $F_{\sgn \otimes \sgn, q-k}$. 

To enable us to succinctly discuss all of the above cases we will associate a character $\mu$ and $\underline{\mu}$ to each principal series modules. Note $\delta$ is a $K$-character and $\mu,\underline{\mu}$ are characters of $M$.
\begin{table}[ht]
$$
\begin{tabu}{cccc}

    G = Sp_{2n}(\mathbb{R}) &  X_{\delta}^\nu, \delta  = (\triv)^k \otimes (\sgn)^{n-k}& \mu = \triv& \underline{\mu} =\det\\
    G = O(p,q) & X_{\delta}^\nu, \delta = \triv_{p-q} \otimes (\triv)^k \otimes (\sgn)^{q-k}& \mu = \triv \otimes \det& \underline{\mu} =\triv \otimes \triv\\
    G = O(p,q) & X_{\delta}^\nu, \delta = \det_{p-q} \otimes (\triv)^k \otimes (\sgn)^{q-k}& \mu = \det \otimes \triv& \underline{\mu} =\sgn \otimes \sgn\\
\end{tabu}$$
\caption{Characters $\mu$, $\underline{\mu}$ associated to particular principal series module.}
\label{muvalues}
\end{table}

\begin{lemma}\label{basisofF} Let $G = Sp_{2n}(\mathbb{R})$ or $G = O(p,q)$. 
If $X_{\delta}^\nu$ is a minimal principal series module, then $F_{\mu,k}(X_{\delta}^\nu)$ and $F_{\underline{\mu},n-k}(X_{\delta}^\nu)$ are finite dimensional.
 with dimensions: 
 $$\Dim(F_{\mu,k}(X_{\delta^k}^\nu))=  k!2^k = |W(B_k)|,$$
Similarly,
$$\Dim(F_{\underline{\mu},n-k}(X_{\delta^k}^\nu))= (n-k)!2^{n-k}=|W(B_{n-k})|.$$
\end{lemma}
This is an extension of \cite[Lemma 2.5.1]{CT11} to non-spherical principal series modules and we use the same arguments. 

\begin{proof} We explicitly calculate a basis for 
$$F_{\mu,k}(X_{\delta}^\nu)=\Hom_K(\mu,X_{\delta}^\nu\otimes V^{\otimes k}).$$
Since $X_{\delta}^\nu$ is an induced module from $\mathbbm{1}_{\delta}^\nu$ and $K$ is a compact group, by Frobenius reciprocity this is equal to,
$$F_{\mu,k}(X_{\delta}^\nu)=\Hom_M(\mu|_M, \mathbbm{1}_{\delta}^\nu\otimes V|_M^{\otimes k}).$$
One can tensor by $\mu^*$ to get a space fixed by $M$, hence

$$F_{\mu,k}(X_{\delta}^\nu)=(\mu^*\otimes  \mathbbm{1}_{\delta}^\nu\otimes V^{\otimes k})^M.$$

We first prove the result for $G= Sp_{2n}(\mathbb{R})$. The module $V$ has basis $\{f_i^{n_i}:i = 1,..., n\text{ and }n_i = \pm 1\}$ and the $j^{th}$ generator of $M$ acts by $-1^{\delta_{ij}}$ on $f_i^{n_i}$.
Therefore if we require $M$ to act trivially on  $u \in X_{\delta_k}^\nu\otimes V^{\otimes k}$ the generators $M_1,...,M_n$ must act by $1$. Let us first calculate all of the elementary tensors in $ X_{\delta^k}^\nu \otimes V^{\otimes k}$ which are fixed by $M$. The generators $M_1,...,M_k$ act by $-1$ on $ \mathbbm{1}_{\delta^k}^\nu$, hence must act by $-1$ on the tensor part contributed by $V^{\otimes k}$. To satisfy this we need to have $f_i^{1}$ or $f_i^{-1}$ feature in the tensor of $u$, for every $i =1,...,k,$. Since there can only be $k$ elements tensored together in $V^{\otimes k}$ then the contribution of $u$ from $V^{\otimes k}$ must be $f_1^{n_1},...,f_k^{n_k}$ in some order. The set of elementary tensors in $V^{\otimes k}$ which feature all the required $f_i$ is the $S_k$ orbit of $f_1\otimes ...\otimes f_k$. Considering not necessarily elementary tensors in $ v \in X_{\delta^k}^\nu \otimes V^{\otimes k}$, 
$$v = \sum x_0 \otimes v_1 \otimes ...\otimes v_k,$$
where $v_i \in \{f_l^{n_l}: l = 1,...,n \text{ and } n_l = \pm1\}$. The $j^{th}$ generator of $M$, $M_j$, acts by $-1^{\delta^{lj}}$ on $ f_l$. Since every elementary tensor in this basis is an eigenvector of the action of $M$ then if $M$ fixes $v = \sum x_0 \otimes v_1 \otimes ...\otimes v_k$ then $M$ fixes each elementary tensor in $v$. Hence every $M$ fixed vector in $X_{\delta^k}^\nu \otimes V^{\otimes k}$ is in the subspace
$$span\left\{\sum_{w\in S_k}  \mathbbm{1}_{\delta^k}^\nu \otimes f_{w(1)}^{n_1}\otimes... \otimes f_{w(k)}^{n_k}: n_i = \pm 1\right\}.$$
The size of the basis is $|S_k| \times 2^k = k! 2^k= |W(B_k)| .$ The proof is almost identical for $\Dim(F_{\Det,n-k}(X_{\delta^k}^\nu)).$ One needs to note that all of the generators of $M$ must act by $-1$ on the $\Det$ isotypic space, since $\Det|_M = \sgn$. Using Frobenius reciprocity one can show,
$$ F_{\Det,n-k}(X_{\delta^k}^\nu)= \Hom_M(\sgn, \delta_k \otimes V^{\otimes n-k}),$$
which has a basis: 

$$ F_{\Det,n-k}(X_{\delta^k}^\nu)=\Span\left\{\sum_{w\in S_{n-k}}  \mathbbm{1}_{\delta^k}^\nu \otimes f_{w(k+1)}^{n_{k+1}}\otimes... f_{w(n)}^{n_n}: n_i = \pm 1\right\} .$$

For $G = O(p,q)$ note that $V|_M = V_{p-q} \bigoplus_{i=1}^{q} \triv\otimes ... \otimes \overbrace{\sgn}^{i^{th}} \otimes ...\otimes\triv$ and $\mu|_M = \triv_{p-q}\otimes \sgn^q$. Recall the notation $f_i^{n_i} = e_{p-i+1} +n_i e_{p+i}$, the vectors $f_i^1$ and $f_i^{-1}$ are the two eigenvectors of $M$ with character $\triv \otimes \triv ...\otimes \overbrace{\sgn}^{i^{th}} \otimes ...\otimes \triv$. I.e. the $i^{th}$ generator of $O(1)^q$ in $M$ acts by $-1$.

 We will prove that $F_{\triv \otimes \sgn, k}(X_{\delta^k_{\triv}}^\nu)$ has basis 
$$\left\{\sum_{w\in S_{k}}  \mathbbm{1}_{\delta^k_{\triv}}^\nu \otimes f_{w(k)}^{n_{1}}\otimes... f_{w(k)}^{n_k}: n_i = \pm 1\right\}.$$
The other four calculations are almost identical. 
Note that this is equivalent to giving a basis for 
$$((\triv \otimes \sgn)|_M \otimes \mathbbm{1}_{\delta^k_{\triv}}^\nu  \otimes V^{k})^M$$
which is equal to, as a vector space,
$$(\mathbbm{1}_{\triv_{p-q} \otimes \sgn^q} \otimes \mathbbm{1}_{\delta^k_{\triv}}^\nu  \otimes (V_{p-q}  \bigoplus \triv ... \otimes \sgn \otimes ...\sgn) ^{k})^M.$$
The vector $\mathbbm{1}_{\triv_{p-q} \otimes \sgn^q} \otimes \mathbbm{1}_{\delta^k_{\triv}}^\nu \otimes f_1 \otimes ... \otimes f_q$ is fixed by $M$ since $O(p-q)$ acts trivially on each tensor. Furthermore for $i =1,...,k$ the  $i^{th}$ generator of $O(1)^q$ in $M$ acts by $-1$ on $ \mathbbm{1}_{\triv_{p-q} \otimes \sgn^q}$, $1$ on $ \mathbbm{1}_{\delta^k_{\triv}}^\nu $, and $-1$ on $f_1 \otimes ... \otimes f_q$. For $i = k+1 ,...q$  the  $i^{th}$ generator of $O(1)^q$ in $M$ acts by $-1$ on $ \mathbbm{1}_{\triv_{p-q} \otimes \sgn^q}$ $-1$ on $ \mathbbm{1}_{\delta^k_{\triv}}^\nu $ and $1$ on $f_1 \otimes ... \otimes f_q$.  Hence every generator of $M$ acts by $1$. An identical argument shows that the orbit of $\mathbbm{1}_{\triv_{p-q} \otimes \sgn^q} \otimes \mathbbm{1}_{\delta^k_{\triv}}^\nu \otimes f_1 \otimes ... \otimes f_q$ by $W(B_q)$ is also fixed. Any elementary tensor fixed by $M$ must be of this form; if it is not, one of the generators will act by $-1$.  Finally suppose that another vector $v$ is fixed by $M$, then $v$ is a sum of elementary tensors which are all eigenvalues for $O(1)^q$, hence every elementary tensor involved must be fixed. This concludes that $v$ is in the span of the vectors 
$$\left\{\sum_{w\in S_{k}}  \mathbbm{1}_{\delta^k_{\triv}}^\nu \otimes f_{w(1)}^{n_{1}}\otimes... f_{w(k)}^{n_k}: n_i = \pm 1\right\}.$$

We state the basis for $F_{\mu,k}$ and $F_{\underline{\mu},n-k}$
Let $G = O(p,q)$
$$
\begin{array}{rcl}
F_{\triv \otimes \det ,k}(X_{\delta^{k}_{\triv}}^\nu) &=&\Span\left\{\sum_{w\in S_{k}}  \mathbbm{1}_{\delta^k_{\triv}}^\nu \otimes f_{w(1)}^{n_{1}}\otimes... f_{w(k)}^{n_k}: n_i = \pm 1\right\},\\
F_{\triv \otimes \triv ,q-k}(X_{\delta^{k}_{\triv}}^\nu) &=& \Span\left\{\sum_{w\in S_{q-k}}  \mathbbm{1}_{\delta^k_{\triv}}^\nu \otimes f_{w(k+1)}^{n_{k+1}}\otimes... f_{w(q)}^{n_q}: n_i = \pm 1\right\},\\
F_{\det \otimes \triv ,k}(X_{\delta^{k}_{\det}}^\nu) &=& \Span\left\{\sum_{w\in S_{k}}  \mathbbm{1}_{\delta^k_{\det}}^\nu \otimes f_{w(1)}^{n_{1}}\otimes... f_{w(k)}^{n_k}: n_i = \pm 1\right\},\\
F_{\det \otimes \det ,q-k}(X_{\delta^{k}_{\det}}^\nu) &=& \Span\left\{\sum_{w\in S_{q-k}}  \mathbbm{1}_{\delta^k_{\det}}^\nu \otimes f_{w(k+1)}^{n_{k+1}}\otimes... f_{w(q)}^{n_q}: n_i = \pm 1\right\}.
\end{array}$$

\end{proof}

\end{section}
\begin{section}{Images of principal series modules}\label{restrictedfunctors}

We write the Type B/C VW-algebra as $\mathfrak{B}_k^\theta$ and omit $m$.

We show that on minimal principal series representations the functors (Definition \ref{functorsdef}) which take admissible $O(p,q)$ or $Sp_{2n}$-modules to $\mathfrak{B}^\theta_k$-modules naturally descend to graded Hecke algebra $\mathbb{H}_k(c)$-modules, for $c$ equal to $0$, $1$ or $\frac{p-q}{2}$. 

In Section \ref{quotientsofbrauer} Lemma \ref{quotienttohecke}, we proved that the type $B/C$ VW-algebra has quotients isomorphic to the Hecke algebra $\mathbb{H}_k(c)$ with parameter $c \in \mathbb{R}$. This  quotient was defined by the relations $e_{i,i+1}= 0 $ and $\theta_k x_k + x_k \theta_k = 2c -2r\theta_k.$ Hence to show that $F_{\mu, k}(X_{\delta}^\nu)$ descends to an $\mathbb{H}_k(c_{\mu})$-module we must prove $e_{i,i+1}=0$ and $\theta_k x_k +x_k \theta_k =2c_{\mu} -2r_{\mu}\theta_k$ as operators on $F_{\mu, k}(X_{\delta}^\nu)$. Similarly to show $F_{\underline{\mu}, n-k}(X_{\delta}^\nu)$ is an $\mathbb{H}_{n-k}(r_{\underline{\mu}})$-module then we must show $e_{i,i+1} = 0 $ and $\theta_{n-k} x_{n-k} + x_{n-k} \theta_{n-k} = 2c_{\underline{\mu}}-2r_{\underline{\mu}}\theta_{n-k}$ on $F_{\underline{\mu}, n-k}(X_{\delta}^\nu)$. The scalars $r_\mu$ and $c_\mu$ will be defined in Table \ref{muconstants}. The arguments of this section are inspired and very similar to \cite[Proposition 2.4.5, Lemma 2.7.2]{CT11}. We extend these results to non-spherical principal series modules. We also utilise an approach from the Brauer algebra perspective not used in \cite{CT11}.

\begin{lemma}{c.f. \cite[2.4.5]{CT11}} On the $\mathfrak{B}_k^\theta$ (resp.  $\mathfrak{B}_{n-k}^\theta$) module $F_{\mu,k}(X_{\delta}^\nu)$ (resp. $F_{\underline{\mu},n-k}(X_{\delta}^\nu$)) the idempotents $e_{i,i+1}$ uniformly act by zero. \end{lemma} 

\begin{proof} Lemma \ref{basisofF} states that the basis of  $F_{\mu,k}(X_{\delta}^\nu)$ is given by $ \mathbbm{1}_{\delta}^\nu \otimes f_{w(1)}^{n_1} \otimes ...\otimes f_{w(k)}^{n_k}$ for $w \in S_k$. The idempotents $e_{i,i+1}$ act by the projection onto the trivial component of $V_i \otimes V_{i+1}$. The trivial component of $V\otimes V$ is one dimensional with spanning vector $\sum_{i=1}^n f_i \wedge f_i'$. The vector $ \mathbbm{1}_{\delta}^\nu \otimes f_{w(1)}^{n_1} \otimes ...\otimes f_{w(k)}^{n_k}$ is in the subspace perpendicular to  $\sum_{i=1}^n f_i \wedge f_i'$ given in Lemma \ref{decompofVtensorV}. Therefore it is in the kernel of the projection $\pr_{i,i+1}$. \end{proof}
Recall Definition \ref{Omega}, $\Omega_{i,j} = \sum_{b \in B} (b)_i \otimes (b^*)_j \in U(g)^{k+1},$ and  $\Omega^{\mathfrak{k}}_{i,j} = \sum_{b \in B\cap \mathfrak{k}} (b)_i \otimes (b^*)_j$. Lemma \ref{braueraction} gives $x_k = \Omega_{0,k} + \Omega_{1,k} +... +\Omega_{k-1,k}$.\\
As operators on $F_{\mu,k}(X_{\delta}^\nu)$:
$$\begin{array}{rcl}
\theta_k x_k + x_k \theta_k &=& \theta_k \sum_{i<k} \Omega_{i,k} +\sum_{i<k} \Omega_{i,k} \theta_k\\[1ex]
&=& (\xi)_k\sum_{i<k} \sum_{b \in B} (b)_i \otimes (b^*)_k + \sum_{i<k} \sum_{b \in B} (b)_i \otimes (b^*)_k  (\xi)_k \\[1ex]
&=& \sum_{i<k} \sum_{b \in B} (b)_i \otimes (\xi b^* + b^* \xi)_k.\end{array}$$
Conjugating by $\xi$ is the Cartan involution. Therefore
$$\xi b^* + b^* \xi =\begin{cases} 0 & \text{ if } b \in \mathfrak{p},\\ 2\xi b^* & \text{ if } b \in \mathfrak{k}.\end{cases}$$
Hence,
$$\begin{array}{rcl}\theta_k x_k + x_k \theta_k &=& 2\sum_{i<k} \sum_{b \in B \cap \mathfrak{k}} (b)_i \otimes (\xi b)_k\\[1ex]
&=& 2 \theta_n \sum_{i<k} \Omega_{i,k}^{\mathfrak{k}}.\end{array}$$
As operators on  $F_{\mu,k}(X_{\delta}^\nu)$ 
$$\theta_k x_k + x_k \theta_k = 2 \theta_k \sum_{i<k} \Omega_{i,k}^{\mathfrak{k}}.$$
Similarly on  $F_{\underline{\mu},n-k}(X_{\delta}^\nu)$ 
$$\theta_{n-k} x_{n-k} + x_{n-k} \theta_{n-k} = 2 \theta_{n-k} \sum_{i<n-k} \Omega_{i,n-k}^{\mathfrak{k}}.$$ 

\begin{lemma}{c.f. \cite[2.7.2]{CT11}} \label{actionofQ} On the $\mathfrak{B}_k^\theta$-module $F_{\mu,k}(X_{\delta}^\nu)$,
 $$\theta_k x_k + x_k \theta_k=  2 \xi\left(\sum_{b\in B \cap \mathfrak{z}} \mu(b) b^* - C^\mathfrak{k}\right)_k,  $$
where $\mathfrak{z}$ is the centre of $\mathfrak{g}$.
 \end{lemma}

\begin{proof} Recall Definition \ref{Omega}, $\Omega_{ij} =\sum_{b \in B \cap \mathfrak{k}} (b)_i \otimes (\xi b)_k.$  Writing $\theta_k x_k + x_k \theta_k$ as operators on $F_{\mu,k}(X_{\delta}^\nu)$,
$$\begin{array}{rcl}\theta_k x_k + x_k \theta_k &=& 2 \theta_k \sum_{i<k} \Omega_{i,k}^{\mathfrak{k}},\\[1ex]
 &=& 2\sum_{i<k} \sum_{b \in B \cap \mathfrak{k}} (b)_i \otimes (\xi b)_k.\end{array}$$
An element $g \in \mathfrak{g}$ acts on the tensor of two modules, $U \otimes W$, as $g \otimes 1 + 1 \otimes g$. Extending this, we can write the action of $b \in U(\mathfrak{g})$ as $\sum_{j =1}^{k+1}(b)_j$ on $X \otimes V^{\otimes k}$. This gives  
$$\theta_k x_k + x_k \theta_k = 2 \theta_k \sum_{b \in B \cap \mathfrak{k}} (b^*)_k b - \sum_{b \in B \cap \mathfrak{k}} (bb^*)_k.$$
By definition $F_{\mu,k}(X_{\delta}^\nu)$ is the $\mu$ isotypic component of $X_{\delta}^\nu$, hence
$$\theta_k x_k + x_k \theta_k= 2 \theta_k \sum_{b \in B \cap \mathfrak{k}} (b^*)_k \mu(b) - \sum_{b \in B \cap \mathfrak{k}} (bb^*)_k.$$

The operator $\sum_{b \in \mathfrak{k}} (b b^*)_k$ is the Casimir operator $C^\mathfrak{k}$ on the $k^{th}$ tensor $V$. We have $\mu(b) = 0 $ unless $b$ is in the centre of $U(\mathfrak{k})$ for any character $\mu$.  Let $\mathfrak{z}$ denote the centre of $\mathfrak{g}$. Therefore,
$$\begin{array}{rcl} \theta_k x_k + x_k \theta_k&=&2 \theta_k \left( \sum_{b \in B \cap \mathfrak{z}} \mu(b)( b^*)_k - (C^\mathfrak{k})_k\right),\\[1ex]
&=& 2\left(\xi(\sum_{b\in B \cap \mathfrak{z}} \mu(b) b^*-C^\mathfrak{k})\right)_k.\end{array}$$ \end{proof}
In order to calculate the action of $ \theta_k x_k + x_k \theta_k$ we must understand the operator 
$$Q_{\mu} =  2 \xi\left(\sum_{b\in B \cap \mathfrak{z}} \mu(b) b^* - C^\mathfrak{k}\right)$$
acting on the $k^{th}$ tensor of $V$.

\begin{lemma} On the $\mathfrak{B}_{n-k}^\theta$-module $F_{\underline{\mu},n-k}(X_{\delta}^\nu)$;
 $$\theta_{n-k} x_{n-k} + x_{n-k} \theta_{n-k}= 2 \left(\xi(\sum_{b\in B \cap \mathfrak{z}} \underline{\mu}(b) b^* - C^\mathfrak{k})\right)_{n-k}. $$
 \end{lemma}

Replacing $\mu$ with $\underline{\mu}$, this follows the same way as Lemma \ref{actionofQ}.  

\begin{lemma} On the module $V$ the operator $Q_{\mu} =  2 \xi\left(\sum_{b\in B \cap \mathfrak{z}} \mu(b) b^*-C^\mathfrak{k} \right)$ (resp. $Q_{\underline{\mu}} = 2\xi\left(\sum_{b\in B \cap \mathfrak{z}} \underline{\mu}(b) b^*-C^\mathfrak{k}\right)$) is equal to 
$2r_\mu + 2c_\mu \xi$ (resp. $2r_{\underline{\mu}} + 2c_{\underline{\mu}} \xi$), where $r_\mu$ and $c_\mu$ are scalars given below. 
\begin{table}[ht]
$$\begin{array}{cccc}
G = Sp_{2n}(\mathbb{R})& \mu = \triv& r_{\triv} = 0 & c_{\triv} = -n\\
G = Sp_{2n}(\mathbb{R})& \underline{\mu} = \det& r_{\det} = 1 & c_{\det} = -n\\
G = O(p,q)& \mu = \triv\otimes \det& r_\mu = \frac{p+q}{2} & c_\mu = \frac{p-q}{2}\\
G = O(p,q)& \mu = \triv\otimes \triv& r_\mu = \frac{p+q}{2} & c_\mu = \frac{p-q}{2}\\
G = O(p,q)& \underline{\mu} = \det\otimes \triv& r_\mu = \frac{p+q}{2} & c_\mu = \frac{p-q}{2}\\
G = O(p,q)& \underline{\mu} = \det\otimes \det& r_\mu = \frac{p+q}{2} & c_\mu = \frac{p-q}{2}\\

\end{array} 
$$
\caption{Values of $c_\mu$ and $r_\mu$ for particular $K$-characters $\mu$.}
\label{muconstants}
\end{table}
In fact for $G = O(p,q)$, $r_\mu$ and $c_\mu$ are independent of $\mu$. 
\end{lemma}
Recall Lemma \ref{hecketypes}, we have isomorphisms: $ \mathbb{H}_k1) \cong \mathbb{H}^{B_k}, \mathbb{H}_k(\frac{1}{2}) \cong \mathbb{H}^{C_k}$ and $\mathbb{H}_k(0)$ is congruent to an extension of the type $D$ graded Hecke algebra $\mathbb{H}^{D_k}$. Hence when $G$ is split, that is $G = O(n+1,n)$ or $Sp_{2n}(\mathbb{R})$ then $c_\mu = 1,\frac{1}{2}$ or $0$ and we obtain correspondences between principal series modules of split real orthogonal Lie groups with  graded Hecke algebras of type $C$ and split real symplectic groups with graded Hecke algebras of type $B$ and $D$.
\begin{proof}We prove the result first for $G= Sp_{2n}(\mathbb{R})$, in this case $\mathfrak{g} = \mathfrak{sp}_{2n}$ and $\mathfrak{k} = \mathfrak{gl}_n$. The Casimir $C^\mathfrak{k}$ acts by the scalar $n$ on $V$. 
The character $\triv$ is zero uniformly on $\mathfrak{k}$ hence $\triv(b) =0$ for all $b$ and there is no contribution from $\sum_{b \in B \cap \mathfrak{z}} \triv(b)b^*.$ For the operator $\sum_{b \in B \cap \mathfrak{z}} \Det(b)b^*$, we note that  the centre of $\mathfrak{k}=\mathfrak{gl}_n(\mathbb{C})$ is the span of the identity matrix, also the character $\Det$ of $U(n)$ differentiated to $\mathfrak{k}$ is the trace character of $\mathfrak{gl}_n$. Taking the spanning vector $\Id_n$ of the centre $\mathfrak{z}$ of $\mathfrak{gl}_n$ then on $V$, $\sum_{b \in B \cap \mathfrak{z}} \Det(b) b^*$ is equal to 
$$\begin{array}{rcl} \sum_{b \in B \cap \mathfrak{z}} \Det(b) b^* &=&   \trace(\Id_n) \Id_n^*\\[1ex]
&=& n \frac{1}{n}\Id_n,\\[1ex]
&=& \Id_n.\end{array}$$
S since $Id_n$ is symmetric, the identity matrix in $U(\mathfrak{k})$ embedded into $\mathfrak{g}$ is 
$$\begin{bmatrix}  0& i \Id_n \\ -i\Id_n &0 \end{bmatrix}.$$
The matrix $\xi$, defined by the Cartan involution of $Sp_{2n}(\mathbb{R})$ is equal to 
$$\xi = \begin{bmatrix} 0 & i\Id_n \\ -i\Id_n &0 \end{bmatrix}.$$
Hence $$\sum_{b\in B \cap \mathfrak{z}} \Det(b) b^* = \xi,$$
as operators on $V$. \end{proof}
 
Now let $G = O(p,q)$ $p+q = 2n+1$, then $\mathfrak{g} = \mathfrak{so}_{2n+1}$ and $\mathfrak{k} = \mathfrak{so}_p \oplus \mathfrak{so}_q$. 

Any character $\mu$ of $K$ differentiated and then restricted to $\mathfrak{z}$ is zero. Hence for any $\mu$, $$\sum_{b \in \mathfrak{z}} \mu(b)b^* = 0.$$ We are left to calculate $C^\mathfrak{k}$ on $V$. $C^{\mathfrak{k}}$ acts by 
$$\begin{bmatrix} p \Id_p & 0 \\ 0 & q \Id_q\end{bmatrix}.$$
 For $G = O(p,q)$ the semisimple element defining $\theta$ is
$$\xi = \begin{bmatrix} \Id_p & 0 \\ 0 &  -\Id_q\end{bmatrix}.$$
Hence for $G= O(p,q)$ 
$$Q_\mu= 2 \xi (\sum_{b \in \mathfrak{z}} \mu(b)b^*-C^{\mathfrak{k}})  = 2\xi(-\frac{p+q}{2}{\Id} - \frac{p-q}{2} \xi) = (q-p) \Id_n - (p+q) \xi .$$

\begin{corollary} For $G = O(p,q)$ or $Sp_{2n}$, consider the  principal series module $X_\delta^\nu$ for particular $\mu$ and $\underline{\mu}$ given in Table \ref{muvalues}. On the $\mathfrak{B}_{k}^\theta$-module $F_{\mu,k}(X_{\delta}^\nu)$, the following equality holds:
 $$\theta_{n-k} x_{n-k} + x_{n-k} \theta_{n-k}= 2r_\mu -2c_{\mu}\theta_{n-k}.$$
Hence  by  Lemma \ref{quotienttohecke}, $F_{\mu,k}(X_{\delta}^\nu)$ is an $\mathbb{H}_k(c_\mu)$-module via the quotient defined by the relations $e_{i,i+1} =0$ and $\theta_{n-k} x_{n-k} + x_{n-k} \theta_{n-k }= 2r_\mu+ 2c_\mu\theta_{n-k}$.
Similarly $F_{\underline{\mu}, n-k}(X_{\delta}^\nu)$ is an $\mathbb{H}_{n-k}(c_{\underline{\mu}})$-module. 
\end{corollary}

We have shown that the image of $X_\delta^\nu$ under the functor $F_{\mu,k}$ naturally descends to a module of the graded Hecke algebra $\mathbb{H}_k(c_\mu)$. 

\begin{theorem} Let $X_{\delta}^\nu$ be a minimal principal series module of $G = Sp_{2n}(\mathbb{R})$ or $O(p,q)$. Let $\mu$ and $\underline{\mu}$ be the particular characters in Table \ref{muvalues} and $r_\mu$, $c_\mu$ be particular scalars in Table \ref{muconstants}.
Let $\pi$ denote the homomorphism from $\mathfrak{B}_k^\theta[m_0,m_1]$ to $\End(F_{\mu,k}(X_{\delta}^\nu))$ in Lemmas \ref{braueraction} and \ref{actionofB}. The graded Hecke algebra $\mathbb{H}_k(c_\mu)$ acts     on $F_{\mu,k}(X_{\delta}^\nu)$, by the homomorphism,
$$
\begin{array}{rcl} \psi: \mathbb{H}_k(c_\mu)& \to& \End(F_{\mu,k}(X_{\delta}^\nu)),\\[1ex]
\epsilon_i &\mapsto &\pi(x_i -r_\mu),\\[1ex]
 s_{i,i+1}& \mapsto  &\pi(s_{i,i+1}),\\[1ex]
s_{\epsilon_i}& \mapsto &\pi(\theta_i).\end{array}$$
Hence $F_{\mu,k}(X_{\delta}^\nu)$, can be considered as an $\mathbb{H}_k(c_
\mu)$-module. 

Let $\underline{\pi}$ denote the homomorphism from $\mathfrak{B}_{n-k}^\theta[m]$ to $\End(F_{\underline{\mu},n-k}(X_{\delta}^\nu))$. The graded Hecke algebra $\mathbb{H}_{n-k}(c_{\underline{\mu}})$ acts on $F_{\underline{\mu},n-k}(X_{\delta}^\nu)$, by the homomorphism,
$$\begin{array}{rcl}\underline{\psi}: \mathbb{H}_{n-k}(c_{\underline{\mu}}) &\to& \End(F_{\underline{\mu},n-k}(X_{\delta}^\nu)),\\[1ex]
\epsilon_i &\mapsto& \underline{\pi}(x_i -r_{\underline{\mu}}),\\[1ex]
 s_{i,i+1} &\mapsto & \underline{\pi}(s_{i,i+1}),\\[1ex]
s_{\epsilon_i}& \mapsto &\underline{\pi}(\theta_i).
\end{array}$$
Hence $F_{\underline{\mu},n-k}(X_{\delta}^\nu)$, can be considered an $\mathbb{H}_{n-k}(c_{\underline{\mu}})$-module. 
\end{theorem} 
It should also be noted that as a $\mathfrak{B}_k^\theta$-module $F_{\mu,k}(X_{\delta}^\nu)$ is essentially an $\mathbb{H}_k(c_\mu)$-module. That is, there is no element in $\mathfrak{B}_k^\theta$ that has a non-trivial action on $F_{\triv,k}(X_{\delta}^\nu)$ that does not correspond to an element in the Hecke algebra. 

For $G=O(n+1,n)$ or $Sp_{2n}(\mathbb{R})$, every principal series module is induced from a character on $M$.
Therefore for split real orthogonal or symplectic groups we can entirely describe the Hecke algebra modules resulting from functors $F_{\mu,k}$ and $F_{\underline{\mu},n-k}$ on principal series modules. Casselman \cite{C78} states that every irreducible representation in $\mathcal{HC}(G)$ is a subrepresentation of a principal series module. Therefore if $X$ is a subrepresentation of $X_\delta^\nu$ then $F_{\mu,k}(X)$ also descends to a Hecke algebra module. 
\begin{theorem} 
Let $G$ be a split real Lie group of type $B$ or $C$. Let $X$ be an irreducible Harish-Chandra $G$-module. Hence $X$ is a subrepresentation of a principal series module $X_\delta^\nu$, then the $\mathfrak{B}_k^\theta$ and $\mathfrak{B}_{n-k}^\theta$-modules 
$$F_{\mu,k}(X) \text{ and } F_{\underline{\mu},n-k}(X)$$
naturally descend to $\mathbb{H}_k$ and $\mathbb{H}_{n-k}$-modules. 
\end{theorem}

\begin{proof} Let $X$ be an irreducible Harish-Chandra module. Casselman's theorem shows that $X$ is a submodule of some principal series module, let $X_\delta^\nu$ be such a principal series modules containing $X$ as a submodule. Note that this principal series module may not be unique. 
Then since $F_{\mu,k}(X)$ is exact and $X$ is a submodule of $X_\delta^\nu$ then $F_{\mu,k}(X)$ is a submodule of $F_{\mu,k}(X_\delta^\nu) $ which is a $\mathbb{H}_k$ module. Therefore $F_{\underline{mu},k}(X)$ is a $\mathbb{H}_k$ module. Similarly for $\underline{\mu}$ and $n-k$.
\end{proof}
Therefore for every Harish Chandra module of $O(n+1,n)$ and $Sp_{2n}(\mathbb{R})$ we can define two corresponding Hecke algebra modules. 

\end{section}

\begin{section}{Principal series modules map to principal series modules}\label{isoclassofX}

In this section we take a closer look at the $\mathbb{H}(c_\mu)$-modules obtained from $X_{\delta}^\nu$ under the functors $F_{\mu, k}$ and $F_{\underline{\mu}, n-k}$. We fully classify these as graded Hecke algebra principal series representations related to $\nu$.

Recall that $\mathbb{H}_k(c)$, defined in \ref{Heckedef} is the graded Hecke algebra associated to $W(B_k)$ with parameter function $\mathbf{c}: \Delta \to \mathbb{C}$ such that 

$$\mathbf{c}_{\epsilon_i -\epsilon_{i+1}}= 1 \text{ and } \mathbf{c}_{2\epsilon_i }= 2c.$$

The algebra $\mathbb{H}_k(c)$ contains the group algebra, $\mathbb{C}[W(B_k)]$, of the hyperoctahedral group. 
Recall the labeling of vectors in $X \otimes V ^{\otimes k}$; we label the tensor product starting at zero. A general elementary tensor in $X\otimes V^{\otimes ^k}$ would be written $x_0\otimes v_1 \otimes v_2\otimes...\otimes v_k.$ 
We begin by restricting to the action of the Weyl group $W(B_k)$ inside $\mathbb{H}(c)$ and computing the resulting $\mathbb{C}[W(B_k)]$-modules isomorphism class. 
Fix a $M$-character $\delta$ and recall the $K$-characters $\mu$ and $\underline{\mu}$ depending on $\delta$ from Table \ref{muvalues}.
\begin{lemma}As a $\mathbb{C}[W(B_k)]$-module 
$$F_{\mu,k}( X_{\delta}^\nu) \cong \mathbb{C}[W(B_k)],$$
and as a $\mathbb{C}[W(B_{n-k})]$-module
$$F_{\underline{\mu},n-k}( X_{\delta}^\nu) \cong \mathbb{C}[W(B_{n-k})].$$
\end{lemma}

\begin{proof} From Lemma \ref{basisofF} we have an explicit basis of $F_{\mu,k}( X_{\delta}^\nu)$;
$$\Hom_K(\mu,X_{\delta^{k*}}\otimes V^{\otimes k}) =  span\{\sum_{w\in S_k}  \mathbbm{1}_{\delta}^\nu\otimes f_{w(1)}^{n_1}\otimes... f_{w(k)}^{n_k}\}.$$
The symmetric group $\mathbb{C}[S_k] \subset \mathbb{C}[W(B_k)]$ acts by permuting the tensor product. The reflections in $\mathbb{C}[W(B_k)]$ related to $2\epsilon_i$ act by $id \otimes ...\otimes \theta_i \otimes.. \otimes id$. They take $f_i$ to $f_i'$ on the $i^{th}$ factor of the tensor product.

Take the vector $ \mathbbm{1}_{\delta}^\nu \otimes f_{1}\otimes... \otimes f_{k}$, the $\mathbb{C}[W(B_k)]$ submodule of $F_{\triv,k}( X_{\delta^{k}}^\nu)$ generated by $ \mathbbm{1}_{\delta}\otimes f_{1}\otimes... \otimes f_{k}$ is the subspace spanned by 
$$\{ \mathbbm{1}_{\delta}^\nu\otimes f_{w(1)}^{n_1}\otimes... f_{w(k)}^{n_k}: w \in \mathbb{C}[S_k]\},$$

The only group element of $\mathbb{C}[W(B_k)]$ that fixes $ \mathbbm{1}_{\delta}^\nu \otimes f_1\otimes... \otimes f_k$ is the identity, hence this module has dimension equal to $k! 2^k$, the dimension of $\mathbb{C}[W(B_k)]$. The dimension is equal to the dimension of $F_{\triv,k}(X_{\delta}^\nu)$, therefore we have equality. An isomorphism between the $\mathbb{C}[W(B_k)]$-module $\mathbb{C}[W(B_k)]$ and  $F_{\triv,k}( X_{\delta^{k}}^\nu)$ can be defined by sending the identity element $e \in \mathbb{C}[W(B_k)]$ to $ \mathbbm{1}_{\delta}^\nu\otimes f_1 \otimes ..\otimes f_k.$

The decomposition of $F_{\underline{\mu},n-k}( X_{\delta}^\nu)$ follows in exactly the same way, sending $e   \in \mathbb{C}[W(B_{n-k})]$ to $ \mathbbm{1}_{\delta}^\nu \otimes f_{k+1} \otimes ...\otimes f_n.$ \end{proof}

We have a description of $F_{\mu,k}( X_{\delta^{k}}^\nu)$ as a $\mathbb{C}[W(B_k)]$-module. We would like to describe it as an $\mathbb{H}(c_\mu)$-module. The algebra $\mathbb{H}(c_\mu)$ is generated by $\mathbb{C}[W(B_k)]$ and the affine operators $\epsilon_1,...,\epsilon_k$. Our calculation reduces to calculating the action of the affine operators $\epsilon_i$. The operators $\epsilon_i \in S(\mathfrak{a}_k)$ act on $X_{\delta}^\nu \otimes V^{\otimes k}$ by 
$$\sum_{0<j<i\leq n} \Omega_{ji} + r_\mu.$$

We define principal series representations for $\mathbb{H}_k(c)$. Then we show that the image of $X_{\delta}^\nu$ is isomorphic to a principal series representation defined by a particular character. 

The subspace $\mathfrak{a}_k\subset \mathfrak{a}$ defined in Example \ref{symplecticexample2} is a dimension $k$ subspace of $\mathfrak{a}$. 
\begin{definition}\cite{KR02} Let $\lambda$ be a character for $S(\mathfrak{a}_k) \subset \mathbb{H}_k(c_\mu)$, we define a principal series representation $X(\lambda)$ for $\mathbb{H}_k(c_\mu)$:
$$X(\lambda) = \Ind_{S(\mathfrak{a}_k)}^{\mathbb{H}_k(c)} \lambda.$$
\end{definition}
We write $\mathbbm{1}_\lambda$ for a fixed vector in the image of the character $\lambda:S(\mathfrak{a}) \to \mathbb{C}.$ 
The symmetric algebra $S(\mathfrak{a}_k)$ is generated by the affine operators $\epsilon_1,...,\epsilon_k.$ The principal series representation can be described as a representation generated by, $\mathbbm{1}_\lambda$, a $\mathbb{C}[W(B_k)]$-cyclic vector on which $\epsilon_i$ acts by the scalar $\lambda(\epsilon_i)$. 
We prove that the $\mathbb{C}[W(B_k)]$-module, $F_{\mu,k}( X_{\delta^{k}}^\nu)$ is as a  $\mathbb{H}_k(c_\mu)$-module isomorphic to a principal series module for the correct character $\lambda$.

We fix a specific basis for $\mathfrak{sp}_{2n}$ and $\mathfrak{so}(p,q)$. Since the operators $\Omega_{ij}\in U(\mathfrak{g})^{k+1}$ are defined in terms of, although independent of, a basis for $\mathfrak{g}$. This basis allows us to explicitly calculate $\Omega_{0j}$. It should be emphasized that the following basis is a decomposition of $\mathfrak{g}$ into reduced root spaces under the adjoint action of $\mathfrak{a}$. Recall that $\mathfrak{a}\subset \mathfrak{sp}_{2n}(\mathbb{R})$ is 
$$\left\{ \begin{bmatrix} 0 & B \\B & 0\end{bmatrix}: B \text{ is diagonal}\right \}.$$

\begin{definition}\label{fixedbasis} Recall the decomposition of the Lie algebra $\mathfrak{g}_0$ as 
$$\mathfrak{g}_0 = \mathfrak{n}_0^+ \oplus \mathfrak{a}_0 \oplus \mathfrak{n}_0^-,$$
where $\mathfrak{a}_0$ is the maximal abelian subalgebra of $\mathfrak{p}_0$ and $\mathfrak{n}_0^+$ is the span of the positive root spaces with respect to the restricted root decomposition.
Let $B_{\mathfrak{n}^+}$,$B_{\mathfrak{{n}^-}}$,$B_{\mathfrak{a}}$ be fixed bases for $\mathfrak{n}_0^+$,$\mathfrak{n}_0^-$ and $\mathfrak{a}_0$. The restricted roots $\Sigma$ are $ \pm \epsilon_i \pm \epsilon_j$,$ \pm \epsilon_j$. We will denote a vector in the positive root space $\lambda\in \Sigma^+$ by $n_\lambda$ and the negative root space will be $\hat{n}_\lambda$. For example $n_{\epsilon_i - \epsilon_j}$ for $i<j$ is in $\mathfrak{n}^+$. And $\hat{n}_{\epsilon_i - \epsilon_j} \in \mathfrak{n}_0^-$. We will scale $\hat{n}_\lambda$ such that 
$$\hat{n}_\lambda = n_{-\lambda} = \theta (n_{\lambda}).$$ Hence $n_\lambda + \hat{n_\lambda}$ is $\theta$-invariant and hence in $\mathfrak{k}$.
\end{definition}
\begin{definition}\label{symplecticbasis} 

For $1 \leq s,t\leq n$, the matrix $E_{s,t}$  is the matrix with a 1 in the $s,t$ position and zero elsewhere. Let $i < j$. Set
$$\begin{array}{rcl}
n_{\epsilon_i-\epsilon_j}&=& E_{i,j} +E_{i,n+j} - E_{j,i} + E_{j,n+i} + E_{n+i,j} +E_{n+i,n+j}+E_{n+j,i}-E_{n+j,n+i},\\[1ex]

 \hat{n}_{\epsilon_i-\epsilon_j}&=& -E_{i,j} +E_{i,n+j} + E_{j,i} + E_{j,n+i} + E_{n+i,j} -E_{n+i,n+j}+E_{n+j,i}+E_{n+j,n+i},\\[1ex]
n_{\epsilon_i+\epsilon_j}&=& -E_{i,j} +E_{i,n+j} - E_{j,i} + E_{j,n+i} - E_{n+i,j} +E_{n+i,n+j}-E_{n+j,i}+E_{n+j,n+i},\\[1ex]
\hat{n}_{\epsilon_i+\epsilon_j}&=& E_{i,j} +E_{i,n+j} +E_{j,i} + E_{j,n+i} - E_{n+i,j}-E_{n+i,n+j}-E_{n+j,i}-E_{n+j,n+i},\\[1ex]
n_{\epsilon_i}&=& E_{i,i} -E_{i,n+i} + E_{n+i,i}-E_{n+i,n+i},\\[1ex]
\hat{n}_{\epsilon_i}&=& -E_{i,i} -E_{i,n+i} + E_{n+i,i}+E_{n+i,n+i},\\[1ex]
a_{\epsilon_i} &= &E_{i,n+1} + E_{n+i,i}.\end{array}$$

These vectors give a reduced root space decomposition for $\mathfrak{sp}_{2n}(\mathbb{R}) =  \mathfrak{n}_0^+\oplus\mathfrak{a}_0 \oplus \mathfrak{n}_0^-$ where $a_{\epsilon_i}\in \mathfrak{a}_0$, $n \in \mathfrak{n}_0^+$ and $\hat{n}\in \mathfrak{n}_0^-$. \end{definition}

\begin{example} Let $\mathfrak{g} = \mathfrak{sp}_4.$ We give the basis given in Definition \ref{symplecticbasis} for $\mathfrak{n}^+$,

$$
\begin{array}{rcl} n_{\epsilon_1-\epsilon_2}&= &\begin{bmatrix} 0 & 1 & 0 & 1\\ -1 & 0 & 1 & 0 \\ 0 & 1 & 0 & 1 \\ 1 & 0 & -1 & 0 \end{bmatrix},\\
\vspace{.1cm}\\
n_{\epsilon_1+\epsilon_2}& =& \begin{bmatrix} 0 & -1 & 0 & 1\\ -1 & 0 & 1 & 0 \\ 0 &- 1 & 0 & 1 \\ -1 & 0 & 1 & 0 \end{bmatrix},\\
\vspace{.1cm}\\
n_{\epsilon_1}& = &\begin{bmatrix} 1& 0 & -1 & 0\\ 0 & 0 & 0 & 0 \\ 1 &0 & -1 & 0 \\ 0 & 0 & 0 & 0 \end{bmatrix},\\
\vspace{.1cm}\\
n_{\epsilon_2}& =& \begin{bmatrix} 0 & 0 & 0 & 0 \\0& 1 & 0 & 1\\ 0 & 0 & 0 & 0 \\ 0 &-1 & 0 & -1 \end{bmatrix}.\\\end{array}$$

\end{example}

\begin{definition}\label{orthogonalbasis} Let $\mathfrak{g}_0= \mathfrak{so}(p,q)$ we follow \cite[VI, pg. 371 Example $\mathfrak{so}(p,q)$]{K96}. 

$$\begin{array}{rcl}
n_{\epsilon_i - \epsilon_j} &= &E_{p-j+1,p-i+1} + E_{p-j+1,p+i} - E_{p-i+1,p-j+1} + E_{p-i+1,p+j}\\&&-E_{p+i,p-j+1}-E_{p+i, p+j} -E_{p+j,p-i+1} + E_{p+j, p +i},\\[1ex]
n_{\epsilon_i + \epsilon_j} &=& E_{p-j+1,p-i+1} - E_{p-j+1,p+i} - E_{p-i+1,p-j+1} + E_{p-i+1,p+j}\\&&-E_{p+i,p-j+1}+E_{p+i, p+j} +E_{p+j,p-i+1} - E_{p+j, p +i},\\[1ex] 
$$\hat{n}_{\epsilon_i - \epsilon_j} &= &E_{p-j+1,p-i+1} - E_{p-j+1,p+i} - E_{p-i+1,p-j+1} - E_{p-i+1,p+j}\\&&+E_{p+i,p-j+1}-E_{p+i, p+j} +E_{p+j,p-i+1} + E_{p+j, p +i},\\[1ex]
n_{\epsilon_i + \epsilon_j} &= &E_{p-j+1,p-i+1} + E_{p-j+1,p+i} - E_{p-i+1,p-j+1} - E_{p-i+1,p+j}\\&&+E_{p+i,p-j+1}+E_{p+i, p+j} -E_{p+j,p-i+1} - E_{p+j, p +i}.\end{array}$$
The root space for $\epsilon_i$ is $p-q$ dimensional. Let $l = 1,...,p-q$ then 
$$n^l_{\epsilon_i} = E_{l,p-i+1} - E_{l,p+i} -E_{p-i+1,l}-E_{p+i,l}.$$
Finally 
$$a_{\epsilon_i} = E_{p-i+1,p+i} +E_{p+i,p-i+1}.$$
 \end{definition}

\begin{example} Let $\mathfrak{g}_0 = \mathfrak{so}(3,2).$ We give the basis given in Definition \ref{orthogonalbasis} for $\mathfrak{n}_0^+$,

$$
\begin{array}{rcl} n_{\epsilon_1-\epsilon_2}&= &\begin{bmatrix}0&0&0&0&0 \\0& 0 & 1 & 1 & 0\\ 0&-1 & 0 & 0 & 1 \\ 0&-1 & 0 & 0 & -1 \\0& 0 & -1 &1 & 0 \end{bmatrix},\\
\vspace{.1cm}\\
n_{\epsilon_1+\epsilon_2}&=&\begin{bmatrix}0&0&0&0&0 \\0& 0 & 1 & -1 & 0\\ 0&-1 & 0 & 0 & 1 \\ 0&-1 & 0 & 0 & 1 \\0& 0 & 1 &-1 & 0 \end{bmatrix},\\
\vspace{.1cm}\\
n_{\epsilon_1}& =&\begin{bmatrix}0&0&1&-1&0 \\0& 0 & 0 &0 & 0\\ -1&0 & 0 & 0 & 0 \\ -1&0 & 0 & 0 & 0 \\0& 0 & 0 &0 & 0 \end{bmatrix},\\
\vspace{.1cm}\\
n_{\epsilon_2}& =&\begin{bmatrix}0&1&0&0&-1 \\-1& 0 & 0 &0 & 0\\ 0&0 & 0 & 0 & 0 \\ 0&0 & 0 & 0 & 0 \\-1& 0 & 0 &0 & 0 \end{bmatrix}.\\\end{array}$$

\end{example}
\begin{lemma}\label{explicitaction}  For $G= Sp_{2n}$, recall the basis $f_i = e_i + e_{n+i}, f_i' = e_i - e_{n+i}$ of $V=\mathbb{C}^{2n}$. For $G= O(p,q)$ we recall that $f_i = e_{p-i +1} + e_{p+i}$, $f_i' = e_{p-i+1} - e_{p+i}$. 
Then by left multiplication of the given matrix in Definitions \ref{symplecticbasis} and \ref{orthogonalbasis} we can calculate the following actions on $f_i$:
$$n_{\epsilon_i+\epsilon_j}(f_k) = 0 \text{ for all } k,$$
$$n_{\epsilon_i+\epsilon_j}(f_k') = \begin{cases}2f_j' \text{ if } f_k' = f_j',\\
0 \text{ otherwise.}\end{cases}$$

$$n_{\epsilon_i}(f_k) =0 \text{ for all } k,$$
$$n_{\epsilon_i}(f_k') =\begin{cases}2 f_k & \text{ if } f_k' = f_i' ,\\ 0 & \text{ otherwise},\end{cases}$$
$$n_{\epsilon_i-\epsilon_j}(f_k) = \begin{cases}2 f_i & \text{ if } f_k = f_j, \\ 0 & \text{ otherwise},\end{cases}$$
$$n_{\epsilon_i-\epsilon_j}(f_k') = \begin{cases}2 f_j' & \text{ if } f_k' = f_i, \\ 0 & \text{ otherwise},\end{cases}$$
$$(n_{\epsilon_i-\epsilon_j}+ \hat{n}_{\epsilon_i-\epsilon_j})(f_k) = \begin{cases} f_i & \text{ if } f_k = f_j, \\ -f_j & \text{ if } f_k = f_i,\\ 0 & \text{ otherwise}.\end{cases}$$
\end{lemma}\begin{proof}
This follows from left multiplication of the elements of $\mathfrak{sp}_{2n}$ and $\mathfrak{so}(p,q)$ on the defining module  $V$ with elements $f_i$ and $f_i'$ in the basis of $V$. \end{proof} 

To prove that the $\mathbb{C}[W(B_k)]$-module is in fact isomorphic to a principal series $\mathbb{H}_k(c_\mu)$-module we need to find a $\mathbb{C}[W(B_k)]$ cyclic vector such that the $\epsilon_i$ act by scalars on this cyclic vector. The cyclic vector is $ \mathbbm{1}_{\delta}^\nu \otimes f_{1}\otimes...\otimes f_{k}$.

\begin{lemma}\label{actionofOmegak} On the vector $\mathbbm{1}_{\delta}^\nu\otimes f_{1}\otimes...\otimes f_{k}$ the operator $\Omega_{0l}$ acts by
$$\nu(\epsilon_l)  - \sum_{t<l}(s_{tl} + id)- \sum_{t>l} id .$$
\end{lemma}
\begin{proof}
Recall that $\Omega_{0l}$ is defined to be $\sum_{b \in B}(b)_0 \otimes (b^*)_l$ for a given basis $B$ of $\mathfrak{g}_0$. We choose to use the fixed basis defined in Definition \ref{fixedbasis}.

The subspace $\mathfrak{a}$ is the Lie algebra of the subgroup $A\subset G$. The basis of $\mathfrak{a}$ defined in \ref{symplecticbasis} and \ref{orthogonalbasis} is such that $a_{\epsilon_i}(f_j) = \delta_{ij}f_j$. Furthermore $a_{\epsilon_i}$ acts on the cyclic vector $\mathbbm{1}_{\delta}^\nu$ of $X_{\delta}^\nu$ by $\nu(x_i)$.   Therefore the contribution from $\mathfrak{a} \subset \mathfrak{g}$ is 
$$(a_{\epsilon_i})_0\otimes (a_{\epsilon_i})_l = \delta_{il} \nu(x_{i_l}).$$

The module $X_{\delta}^\nu$ is induced from the character $\delta \otimes e^\nu \otimes 1$ of $MAN$, which is the trivial character on $N$. The space $\mathfrak{n}_0^+$ is the Lie algebra of $N$. The differential of  the trivial character to $\mathfrak{n}_0^+$ is zero. Therefore $(n)_0$ acts by zero on $\mathbbm{1}_{\delta}^\nu$ for all $n \in \mathfrak{n}_0^+.$ Hence the contribution from $\mathfrak{n}^+$ is:
$$(n)_0\otimes (n)^*_l = 0, \text{for } n \in \mathfrak{n}^+.$$
Since $n\in \mathfrak{n}_0^+$ annihilates $\mathbbm{1}_{\delta}^\nu$ then $(n)_0 \otimes (b)_l = 0$ for any $b \in \mathfrak{g}_{2n}$, $n \in \mathfrak{n}^+$, a fact we will use later in this proof.
The operator $ (\hat{n}_{\epsilon_i+\epsilon_j})^*_l$ is equal to $\frac{1}{2}(n_{\epsilon_i+\epsilon_j})_l$ which is zero on any $f_k$ hence;
$$(\hat{n}_{\epsilon_i+\epsilon_l})_0\otimes (\hat{n}_{\epsilon_i+\epsilon_j})^*_l = 0.$$
Similarly $n_{\epsilon_i}$ is zero on any $f_k$ therefore;
$$(\hat{n}_{\epsilon_i})_0\otimes (\hat{n}_{\epsilon_i}^*)_l=0.$$
The only remaining basis elements to consider are those of the form $n_{\epsilon_i - \epsilon_j}$ from $\mathfrak{n}_0^-\subset \mathfrak{g}_0$. We utilise the trick that as a $K$-module $F_{\mu,k}( X_{\delta^{k}}^\nu)$ is just the $\mu$ isotypic component of $X_{\delta^{k}}^\nu \otimes V^{\otimes k}$.  The contribution from $\hat{n}_{\epsilon_i - \epsilon_j}$ is: 

$$(\hat{n}_{\epsilon_i-\epsilon_j})_0\otimes (\hat{n}_{\epsilon_i-\epsilon_j}^*)_l.$$
We can add the operator $(n_{\epsilon_i-\epsilon_j})_0\otimes (\hat{n}_{\epsilon_i-\epsilon_j}^*)_l$ which since $n_{\epsilon_i-\epsilon_j} \in \mathfrak{n}^+$, by above, acts by zero. Therefore we are not modifying the original operator,

$$ (\hat{n}_{\epsilon_i-\epsilon_j})_0\otimes (\hat{n}_{\epsilon_i-\epsilon_j}^*)_l=\frac{1}{2}( \hat{n}_{\epsilon_i-\epsilon_j}+{n}_{\epsilon_i-\epsilon_j})_0 \otimes ({n}_{\epsilon_i-\epsilon_j})_l.$$
The vector $ \hat{n}_{\epsilon_i-\epsilon_j}+{n}_{\epsilon_i-\epsilon_j}$ is $\theta$-invariant, hence is in $\mathfrak{k}$. Recall that for $k \in \mathfrak{k}$ acting on the tensor $X \otimes V^{\otimes k}$ that $k = \sum_{i =0}^k (k)_i$. Since we are working with the $\mu$-isotypic space, we replace $\hat{n}_{\epsilon_i-\epsilon_j}+{n}_{\epsilon_i-\epsilon_j} \in \mathfrak{k}$ by $\mu(\hat{n}_{\epsilon_i-\epsilon_j}+{n}_{\epsilon_i-\epsilon_j})$ and subtract the difference to find,

$$ (\hat{n}_{\epsilon_i-\epsilon_j})_0\otimes (\hat{n}_{\epsilon_i-\epsilon_j}^*)_l$$
$$=\frac{1}{2}\mu( \hat{n}_{\epsilon_i-\epsilon_j}-{n}_{\epsilon_i-\epsilon_j}) \otimes ({n}_{\epsilon_i-\epsilon_j})_l -\frac{1}{2} \sum_{m > 0} ( \hat{n}_{\epsilon_i-\epsilon_j}-{n}_{\epsilon_i-\epsilon_j})_m \otimes ({n}_{\epsilon_i-\epsilon_j})_l.$$
The character $\mu$ (or $\underline{\mu}$) differentiated to $\mathfrak{a}$ is zero (or the trace character) hence $\mu( \hat{n}_{\epsilon_i-\epsilon_j}+{n}_{\epsilon_i-\epsilon_j}) = 0$. Lemma \ref{explicitaction} gives the explicit action of $n_{\epsilon_i - \epsilon_j}$ on $f_k$, using this one can determine the action;
$$(\hat{n}_{\epsilon_i-\epsilon_j})_0\otimes (\hat{n}_{\epsilon_i-\epsilon_j}^*)_l=\begin{cases} -s_{tl}-id & \text{if } f_{i_t} = f_i \text{ and } f_{i_l} = f_j, \\
-id & \text{ if } f_{i_l} = f_i,\\
 0 & \text{otherwise}.\end{cases}$$
The only non-zero terms are contributed by $a_{\epsilon_l}$, and $\hat{n}_{\epsilon_i - \epsilon_l}$ and  $\hat{n}_{\epsilon_l - \epsilon_i}$. Which act, on the cyclic vector, by $\nu(\epsilon_l)$,  $-s_{tl} - id$ and $-id$ respectively. Summing these up gives,

$$\Omega_{0l} =\nu(\epsilon_l) - \sum_{t<l}(s_{tl} + id) - \sum_{t>l} id,$$
on the $\mathbb{C}[W(B_k)]$ cyclic vector $ \mathbbm{1}_{\delta}^\nu\otimes f_1\otimes...\otimes f_k$.
\end{proof}

The equivalent statement for $F_{\underline{\mu},n-k}(X_{\delta}^\nu)$ is below. It follows from the proof of Lemma \ref{actionofOmegak}.
\begin{lemma}On the vector $ \mathbbm{1}_{\delta}^\nu\otimes f_{k+1}\otimes...\otimes f_{n}$ the operator $\Omega_{0l}$ acts by
$$\nu(\epsilon_{k+l})  - \sum_{t<l}(s_{k+t,k+l} + id)- \sum_{t>l} id ,$$
for $l = 1,...,n-k.$
\end{lemma}

\begin{corollary}\label{corollaryactiontriv} The operator $\epsilon_l = \sum_{i<l} \Omega_{il} + n$ acts by the scalar $\nu(\epsilon_l) $ on the vector $ \mathbbm{1}_{\delta}^\nu\otimes f_1\otimes...\otimes f_k.$\end{corollary}

\begin{proof} This follows from the fact that $\Omega_{0l}$ acts by $\nu(\epsilon_l)-n- \sum_{t<l}s_{tl}$ and, by Lemma \ref{omegatranspositions},  $\sum_{t=1}^{l-1}\Omega_{tl}$ acts by $  \sum_{t<l}s_{tl}$ on $\mathbbm{1}_{\delta}^\nu \otimes f_1 \otimes ...\otimes f_k$.
\end{proof}
\begin{corollary} \label{corollaryactiondet}The operator $\epsilon_l = \sum_{i<l} \Omega_{il} + n$ acts by the scalar $\nu(\epsilon_{k+l}) $ on the vector $ \mathbbm{1}_{\delta}^\nu\otimes f_{k+1}\otimes...\otimes f_{n}.$\end{corollary}

\begin{definition}Example \ref{symplecticexample2} defines subspaces $\mathfrak{a}_k$ and $\bar{\mathfrak{a}}_{n-k}$ of $\mathfrak{a}$ such that 
$$\mathfrak{a}= \mathfrak{a}_k \oplus \bar{\mathfrak{a}}_{n-k}.$$
Let  $\nu$ be a character of $\mathfrak{a}$.  
Define $\nu_k$ to be the restricted character $$\nu|_{\mathfrak{a}_k}$$ and $\bar{\nu}_{n-k}$ to be $\nu|_{\bar{\mathfrak{a}}_{n-k}}.$

\end{definition}
For a principal series module $X_\delta^\nu$ we have shown that as a $W(B_k)$-module $F_{\mu,k}(X_\delta^\nu)$ is isomorphic to $\mathbb{C}[W(B_k)]$ and as a Hecke algebra module it is a principal series module induced from a character of $S(V) \subset \mathbb{H}_k(c_\mu)$.

\principaltoprincipalthm
For spherical principal series representations, this recovers the results of \cite[Theorem 3.0.4]{CT11}. 

\begin{proof}
One defines an isomorphism by taking the given cyclic vector $  \mathbbm{1}_{\delta}^\nu\otimes f_1\otimes...\otimes f_k \in F_{\triv,k}( X_{\delta^{k}}^\nu)$ to the cyclic vector $\mathbbm{1}_{\nu_k}$ of $X(\nu_k)$. Both  vectors are $\mathbb{C}[W(B_k)]$ cyclic. By Corollary \ref{corollaryactiontriv} the affine operators $\epsilon_i$ act on both vectors by $\nu_k(\epsilon_i) $, therefore this is a well-defined isomorphism.

Lemma \ref{basisofF} gives a basis of $F_{\Det,n-k}( X_{\delta^{k}}^\nu)$:
 $$\{ \mathbbm{1}_{\delta}^\nu\otimes f_{w(1)+k}^{n_1}\otimes ...\otimes f_{w(n-k) +k}^{n_{n-k}}: w \in S_{n-k}\}.$$
For $F_{\underline{\mu},k}(X_\delta^\nu)$ and $X(\bar{\nu}_{n-k})$, both modules are $\mathbb{C}[W(B_{n-k})]$ cyclic and Corollary \ref{corollaryactiondet} shows that the affine operators $\epsilon_i$ for $i = 1,...,n-k,$ act by the same scalar on on the cyclic vector $ \mathbbm{1}_{\delta}^\nu\otimes f_{k+1}\otimes ...\otimes f_{n}$ and $\mathbbm{1}_{\bar{\nu}_{n-k}}$, respectively. 
\end{proof}

 Casselman's theorem \cite{C78} states that every irreducible representation in $\mathcal{HC}(G)$ is a subrepresentation of a principal series module. If $G$ is a split real orthogonal or symplectic group then $M$ is abelian and every principal series module is induced from a character.

\splitthm

\end{section}
\begin{section}{Hermitian forms} \label{hermitian}

In this section we define two star operations on $\mathfrak{B}_k^\theta $. Through the quotients defined in Lemma \ref{quotienttohecke} these star operations descend to the usual star operations on the graded Hecke algebras $\mathbb{H}_k(c)$ \cite{BC15}.
We then show that a Harish-Chandra module $X\in\mathcal{HC}(G)$  with invariant Hermitian form gets mapped, by $F_{\mu,k}$, to a $\mathfrak{B}_k^\theta$-module with invariant Hermitian form. This extends the results in \cite{CT11} to any Harish-Chandra module.  Furthermore, if $X$ is a unitary module, then it maps to a unitary module for $\mathfrak{B}_k^\theta$. 
In this section we assume that $\mu$ is a character of $K$. 
\begin{definition} Let $G$ be $O(p,q)$ $p+q = 2n+1$ or $Sp_{2n}(\mathbb{R})$, let $\mathfrak{g}_0$  be its Lie algebra,  with complexification $\mathfrak{g} = \mathfrak{k}\oplus\mathfrak{p}$. Conjugation   $\bar{}:\mathfrak{g} \to \mathfrak{g}$ is defined by the real form $\mathfrak{g}_0$.
Define the star operation as the conjugate linear map $^*: \mathfrak{g} \to \mathfrak{g}$ such that:
$$g^* =  - \bar{g} \text{ for all } g \in \mathfrak{g}.$$
Define the operation $^\bullet: \mathfrak{g} \to \mathfrak{g}$ by:
$$p^\bullet =   \bar{p} \text{ for all } p \in \mathfrak{p}.$$
$$k^\bullet =   -\bar{k} \text{ for all } k \in \mathfrak{k}.$$
\end{definition}

Recall Definition \ref{Heckedef} of the Hecke algebra $\mathbb{H}_k(c)$. We define the Drinfeld presentation of $\mathbb{H}_k(c)$. 
\begin{definition} Let $R$ be a  root system with pairing $\langle, \rangle : V \times V \to \mathbb{C}$, simple roots $\delta$, and a parameter function $\textbf{c}: \Delta \to \mathbb{C}$. Denote the Weyl group of $R$ by $W(R)$. The Drinfeld Hecke algebra $\mathfrak{H}^R(\textbf{c})$ is a quotient of the algebra 
$$ \mathbb{C}[W(R)]\rtimes T(V),$$
by the relations
$$w \tilde{\alpha} w^{-1} = \widetilde{w(\alpha)} \text{ for all } w \in W(R), \alpha \in V,$$
$$[\tilde{\alpha},\tilde{\beta}] = \sum_{\gamma, \delta \in \Delta} \textbf{c}(\gamma) \textbf{c}(\delta)(\langle \tilde{\alpha} , \gamma\rangle \langle \tilde{\beta}, \delta\rangle-\langle \tilde{\beta} , \gamma\rangle \langle \tilde{\alpha}, \delta\rangle) s_\gamma s_\delta.$$
\end{definition} 

\begin{lemma} The Drinfeld Hecke algebra and the graded Hecke algebra are defined by a root system and a parameter on simple roots. If the defining root systems and parameters are equal then these algebras are isomorphic.\end{lemma} 
\begin{proof} One defines an isomorphism $\phi: \in \mathbb{H}^R(c)$ to $\mathfrak{H}^R(\textbf{c})$ by 
$$\phi (\alpha -\frac{1}{2} \sum_{\gamma \in \Delta} \textbf{c}(\gamma) \langle \gamma, \alpha \rangle s_\gamma )= \tilde{\alpha},$$
$$\phi (w) = w, \hspace{1cm}\forall w \in W(R).$$
\end{proof}
Given that the graded Hecke algebra and the Drinfeld Hecke algebra are isomorphic we omit the different notation and denote the graded Hecke algebra by $\mathbb{H}^R(\textbf{c})$. We uniformly denote a generator in the Drinfeld presentation by $\tilde{\alpha}$ and $\alpha$ denotes a generator in the Lusztig presentation (Definition \ref{Heckedef}).

\begin{definition} \label{Heckestar} Let $*: \mathbb{H}_k(c) \to \mathbb{H}_k(c)$ be the antihomomorphism such that:
$$\tilde{\alpha}^* =  \overline{-\tilde{\alpha}} \text{ for all } \alpha \in\tilde{\mathfrak{t}},$$
$$g^* = g^{-1} \text{ for all } g \in W(B_k).$$

Let $^\bullet: \mathbb{H}_k(c) \to \mathbb{H}_k(c)$ be the antihomomorphism such that:
$$\alpha^\bullet=  \overline{\alpha} \text{ for all } \alpha \in \mathfrak{t}  \hspace{0.2cm}(\text{equivalently } \tilde{\alpha}^\bullet =\overline{ \tilde{\alpha})} ,$$
$$g^\bullet = g^{-1} \text{ for all } g \in W(B_k).$$
Here $\overline{v}$ is the complex conjugate of $v$. 

\end{definition}

 Let $w_0$ be the longest element in $W(B_k)$, it is an involution and is generated by $k^2$ simple reflections. It is in the centre of $W(B_k)$. On the space of roots $w_0$ acts by $-1$.
\begin{lemma} The longest element $w_0$ can be written as 
$$w_0 = \theta_1 \theta_2 ...\theta_k.$$\end{lemma}
It is well known that the longest element $w_0$ relates the two star operations $^*$ and $^\bullet$ in $\mathbb{H}_k(c)$.
\begin{lemma}
$$h^* = w_0 h^\bullet w_0 \text{ for all } h \in \mathbb{H}_k(c).$$
\end{lemma}

\begin{lemma} The longest element $w_0$ is central in the finite Brauer algebra $Br_k[m]$. \end{lemma}
\begin{proof} The element $w_0$ is central in $W(B_k)$, therefore it is sufficient to prove that $w_0$ commutes with the idempotents $e_{i,i+1}$.
The reflections $\theta_l$ commute with $e_{i,j}$
$$[e_{i,j},\theta_l] = 0 \text{ for all } i,j,l.$$

We have,
$$w_0 e_{i,j} w_0 = \theta_1 ...\theta_k e_{i,j}\theta_k..\theta_1= e_{i,j}.$$
Hence $w_0$ is central in the finite Brauer algebra. 

\end{proof}
Since $w_0 = \theta_1 \theta_2 ...\theta_k$ then as an operator on $X\otimes V^{\otimes k}$ 
$$\pi(w_0) =(\xi)_1 (\xi)_2 ...(\xi)_k = id \otimes \xi \otimes \xi ...\otimes \xi.$$
We calculate how $w_0$ and $\Omega_{ij}, \Omega_{ij}^\mathfrak{k}, \Omega_{ij}^\mathfrak{p}$ interact.

\begin{lemma} As operators on $X \otimes V^{\otimes k}$,

$$w_0(\Omega_{ij}^\mathfrak{k})w_0 =\Omega_{ij}^\mathfrak{k} \text{ for all }0 \leq i<j\leq n,$$
 $$w_0(\Omega_{ij}^\mathfrak{p})w_0 =\begin{cases} \Omega_{ij}^\mathfrak{p}& \text{ for all } 0<i<j\leq n,\\ -\Omega_{0j}^\mathfrak{p}& \text{ when } i = 0.\end{cases}$$

\end{lemma} 

\begin{proof} Recall $\xi g \xi = \begin{cases} g  & \text{ if } g \in \mathfrak{k},\\ -g &\text{ if } g \in \mathfrak{p}.\end{cases}$ Therefore one finds that $\pi(w_0) = id \otimes \xi \otimes ...\xi$ commutes with $\Omega_{ij}^{\mathfrak{k}}  = \sum_{b \in B \cap \mathfrak{k}} (b)_i \otimes (b^*)_j.$
For $\Omega_{ij}^\mathfrak{p}$ we have:

$$\begin{array}{rcl}
w_0 (\Omega_{ij}^\mathfrak{p})w_0 &=&\left (id \otimes \xi \otimes \xi \otimes ...\otimes\xi\right) \Omega_{ij}^\mathfrak{p}\left( id \otimes \xi \otimes ...\otimes \xi\right),\\[1ex]
&=&\left( id \otimes \xi \otimes \xi \otimes ...\otimes\xi\right) \sum_{b \in B \cap \mathfrak{p}} (b)_i \otimes (b^*)_j \left(id \otimes \xi \otimes ...\otimes \xi \right),\\[1ex]
 &=&\begin{cases} \sum_{b \in B \cap \mathfrak{p}} ( b )_i \otimes (\xi b\xi)_j & \text{ if } i= 0,\\ \sum_{b \in B \cap \mathfrak{p}} (\xi b \xi)_i \otimes (\xi b\xi)_j & \text{ if } i \neq0,\end{cases}\\[1ex]
&=&\begin{cases} \sum_{b \in B \cap \mathfrak{p}} ( b )_i \otimes (-b)_j & \text{ if } i= 0,\\ \sum_{b \in B \cap \mathfrak{p}} (-b)_i \otimes (- b)_j & \text{ if } i \neq0,\end{cases}\\[1ex]
&=&\begin{cases} -\sum_{b \in B \cap \mathfrak{p}} ( b )_i \otimes (b)_j & \text{ if } i= 0,\\ \sum_{b \in B \cap \mathfrak{p}} (b)_i \otimes ( b)_j & \text{ if } i \neq0.\end{cases}
\end{array}$$
\end{proof}

\begin{definition} Let $^\bullet: \mathfrak{B}_k^\theta \to \mathfrak{B}_k^\theta$ be the conjugate linear antihomomorphism defined on the generators as follows:
$$z_i^\bullet =z_i$$
$$g^\bullet = g^{-1} \text{ for } g \in W(B_k)$$
$$e_{i,i+1}^\bullet = e_{i,i+1}.$$

\end{definition} 

\begin{remark} To check this antihomomorphism is well defined one must just check that the relations in Definition \ref{BCbrauerdef} are fixed. \end{remark}

\begin{definition} Let $^*: \mathfrak{B}_k^\theta \to \mathfrak{B}_k^\theta$ by the antihomomorphism such that,
$$b^* = w_0 b^\bullet w_0.$$

\end{definition} 

\begin{remark} Since $w_0$ is central in the finite Brauer algebra then $g^* = g^{-1}$ for $g \in W(B_k)$ and $e_{i,j}^* = e_{i,j}.$\end{remark}

\begin{lemma} \label{quotientofstar} Under the quotients in Lemma \ref{quotienttohecke} the antihomomorphisms $^*: \mathfrak{B}_k^\theta \to\mathfrak{B}_k^\theta$ and $^\bullet: \mathfrak{B}_k^\theta\to\mathfrak{B}_k^\theta$ descend to the antihomomorphisms
$^*: \mathbb{H}_k(c) \to \mathbb{H}_k(c)$ and $^\bullet: \mathbb{H}_k(c) \to \mathbb{H}_k(c)$ respectively.\end{lemma} 

\begin{proof}The operation $^\bullet$ fixes $e_{i,i+1}$ and $$\theta_k z_k + z_k \theta_k = 2c - 2r\theta_k.$$
Therefore   $^\bullet$ on $\mathfrak{B}_k^\theta$ descends to $\mathbb{H}_k(c)$. On the generators of $\mathbb{H}_k(c)$ it fixes the affine generators and is the inverse antihomomorphism on the group $W( B_k)$.  Hence the operation $^\bullet$ on $\mathfrak{B}_k^\theta$ descends to the antihomomorphism $^\bullet$ on $\mathbb{H}_k(c)$. 
Since 
$$h^*= w_0 h^\bullet w_0,$$
in both $\mathfrak{B}_k^\theta$ and $\mathbb{H}_k(c)$ then the star operation $^*$ on $\mathfrak{B}_k^\theta$ descends to $^*$ on $\mathbb{H}_k(c)$.
\end{proof}
We give a new set of generators for $\mathfrak{B}_k^\theta$.
\begin{definition}  Define
$$\tilde{z_i} =\frac{z_i - w_0z_i w_0}{2},\text{ for } i = 1,...,k,$$
then
 $$\mathfrak{B}_k^\theta \cong\langle\tilde{z_i}, s_{j,j+1}, e_{j,j+1},\theta_i\rangle.$$
\end{definition} 
The operators $\tilde{z_i}$ form a Drinfeld type presentation for $\mathfrak{B}_k^\theta$, they descend to the Drinfeld presentation of $\mathbb{H}_k(c)$ under the quotients defined in \ref{quotienttohecke}. As operators on $X \otimes V^{\otimes k}$:
$$\begin{array}{rcl} \pi(\tilde{z_i})& =& \frac{1}{2}\pi(z_i - w_0 z_i w_0)\\[1ex]
&= &\frac{1}{2}\left ( \sum_{j<i} \Omega_{ij} - (\xi)_1(\xi)_2...(\xi))_k \sum_{j <i} \Omega_{ij} (\xi)_1(\xi)_2...(\xi))_k\right ),\\[1ex]
&=& \frac{1}{2} \left ( \sum_{j<i} \Omega_{ij}^\mathfrak{k} + \Omega_{ij}^\mathfrak{p} - (\xi)_1(\xi)_2...(\xi))_k \sum_{j <i} \Omega_{ij}^\mathfrak{k} + \Omega_{ij}^\mathfrak{p} (\xi)_1(\xi)_2...(\xi))_k\right ),\\[1ex]
&=& \frac{1}{2} \left ( \sum_{j<i} \Omega_{ij}^\mathfrak{k} + \Omega_{ij}^\mathfrak{p} + \Omega_{0i}^\mathfrak{p}  - \Omega_{0i}^\mathfrak{k} \sum_{0<j <i} \Omega_{ij}^\mathfrak{k} + \Omega_{ij}^\mathfrak{p}\right ),\\[1ex]
&=& \Omega_{0i}^\mathfrak{p}.\\[1ex]
\end{array}$$

\begin{remark} With this presentation of $\mathfrak{B}_k^\theta$ the operation $^*$ is defined as 
$$\begin{array}{rcl}
\tilde{z_i}^* &=& - \tilde{z_i},\\[1ex]
g^*& =& g^{-1} \text{ for all } g \in w(B_k),\\[1ex]
e_{i,i+1}^* &= &e_{i,i+1}.
\end{array}$$
\end{remark} 

\begin{definition} Let $X$ be a complex vector space, a Hermitian form  $\langle,\rangle_X$ on $X$ is a map $\langle,\rangle_X : X \times X \to \mathbb{C}$ such that 
$$\langle \lambda_1x_1 + \lambda_2x_2,x'\rangle_X = \lambda_1\langle x_1',x'\rangle_X +  \lambda_2\langle x_2',x\rangle_X \text{ for all } x_1,x_2,x' \in X, \lambda_1,\lambda_2 \in \mathbb{C},$$
$$\langle x,\lambda_1x_1' + \lambda_2x_2'\rangle_X = \bar{\lambda}_1\langle x,x_1'\rangle_X +   \bar{\lambda}_2\langle x,x_2'\rangle_X \text{ for all } x_1',x_2',x \in X, \lambda_1,\lambda_2 \in \mathbb{C}.$$

\end{definition}

\begin{definition} Let $X$ be a $\mathcal{HC}(G)$-module.  A Hermitian form $\langle, \rangle_X$ is $*$-invariant if:

$$\langle g(x_1), x_2\rangle_X = \langle x_1, g^*(x_2)\rangle, \text{ for all } x_1,x_2 \in X\text{ and } g \in \mathfrak{g}.$$

\end{definition} 

\begin{definition} Let $U$ be an $\mathbb{H}_k(c)$-module.  A Hermitian form $\langle, \rangle_U$ on $U$ is invariant with respect to $^*$ if:

$$\langle h(x_1), x_2\rangle_X = \langle x_1, h^*(x_2)\rangle, \text{ for all } x_1,x_2 \in U\text{ and } h \in \mathbb{H}_k(c).$$
Similarly for $U$ a $\mathfrak{B}_k^\theta$-module, a Hermitian form $\langle,\rangle_U$ on $U$ is $^*$-invariant if 
$$\langle b(x_1), x_2\rangle_X = \langle x_1, b^*(x_2)\rangle, \text{ for all } x_1,x_2 \in U\text{ and } b \in \mathfrak{B}_k^\theta.$$
\end{definition} 

\begin{definition} A $\mathcal{HC}(G)$-module $X$ is unitary if there exists a   positive definite invariant Hermitian form on $X$.\end{definition} 

Similarly, an $\mathbb{H}_k(c)$-module $U$ is unitary if  $U$ has an invariant positive definite Hermitian form  
and a $\mathfrak{B}_k^\theta$-module is unitary if it has a positive definite invariant Hermitian form.

Recall $V$ is the defining matrix module of $G$. Let $\langle ,\rangle_V$ be a non-degenerate Hermitian form on $V$ such that $$\langle k v_1,v_2 \rangle = \langle v_1, k^{-1}v_2 \rangle \text{ for all } v_1,v_2 \in V, k \in K ,$$
$$\langle p v_1,v_2 \rangle = \langle v_1, pv_2 \rangle \text{ for all } v_1,v_2 \in V, p \in\mathfrak{p} .$$
This makes $V$ unitary with respect to $^\bullet$. 

\begin{definition}{c.f. \cite[(4,.4)]{CT11}}\label{inducedform}
Let $X$ be in $\mathcal{HC}(G)$ with a $*$-invariant Hermitian form $\langle ,\rangle_X$ then we endow $X \otimes V ^{\otimes k}$ with a Hermitian form defined on elementary tensors by $$\langle x_0 \otimes v_1 \otimes ...\otimes v_k , x_0'\otimes v_1'\otimes ...\otimes v_k'\rangle_{X \otimes V^{\otimes k}} = \langle x_0 , x_0' \rangle_X \langle v_1,v_1'\rangle_V ...\langle v_k,v_k'\rangle_V,$$ then extended to a Hermitian form.
For $\mu$ a character of $K$, define a Hermitian form on $F_{\mu,k}(X) = \Hom_K(\mu, X \otimes V^{\otimes k})$ by:
 $$ \langle \phi,\psi \rangle_{F_{\mu,k}} = \langle \phi(1),\psi(1)\rangle_{X \otimes V^{\otimes k}},\text{ for all } \phi, \psi \in \Hom_K(\mu, X \otimes V^{\otimes k}).$$
 \end{definition} 

\begin{remark}\label{unitarytensor} If $X$ is a unitary space then $\langle,\rangle_{X \otimes V^{\otimes k}}$ endows $X \otimes V^k$ as a unitary space. \end{remark}
\begin{lemma} \label{pronVV} Let $V$ be the complex matrix module of $G = O(p,q)$ or $Sp_{2n}(\mathbb{R})$ and $pr_{12}$ be the projection of $V\otimes V$ onto its trivial $G$ submodule. Define  $\langle ,\rangle_{V\otimes V}$ on $V \otimes V$ by 
$$\langle v_1 \otimes v_2, v_1'\otimes v_2' \rangle_{V\otimes V} = \langle v_1 ,v_1'\rangle_V \langle v_2,v_2'\rangle_V,$$
and extend to a Hermitian form. Then 
$$\langle pr_{12} (v_1 \otimes v_2) ,v_1'\otimes v_2' \rangle_{V \otimes V} = \langle v_1 \otimes v_2 , pr_{12} (v_1'\otimes v_2') \rangle_{V \otimes V}.$$
\end{lemma} 

\begin{proof}
It is sufficient to prove that the trivial submodule in $V\otimes V$ and its complement are orthogonal with the form $\langle , \rangle_{V\otimes V}$. The Peter-Weyl Theorem \cite[Theorem 1.12]{PW27} states that a unitary module of a compact group decomposes as an orthogonal direct sum of irreducibles. Considering $V \otimes V$ as a $^\bullet$ unitary $K$-module , we have that the trivial submodule of $V \otimes V$ is orthogonal to its complement with respect to $\langle , \rangle_V$. \end{proof} 

\begin{lemma}\label{unitarypreserved} Suppose $X \in \mathcal{HC}(G)$  with a $^*$-invariant Hermitian form $\langle,\rangle_X$ then $F_{\mu,k}(X) \in \mathfrak{B}_k^\theta$-mod has a $^*$-invariant Hermitian form $\langle, \rangle_{F_{\mu,k}}$. Furthermore if $X$ is unitary then $F_{\mu,k}(X)$ is unitary.  \end{lemma}

Ciubotaru and Trapa \cite{CT11} prove this result for spherical principal series modules mapping to graded Hecke algebras. We extend this to any Harish-Chandra module which requires considering the image as a Type B/C VW-algebra module. 

\begin{proof} We show that the Hermitian form is invariant under the generators $\tilde{z_i} $, $s_{i,i+1}$, $\theta_j$ and $e_{i,i+1}$. For $\tilde{z_i}$, $\tilde{z_i}^* = - \tilde{z_i}$ and $\pi(\tilde{z_i}) = \Omega_{0i}^\mathfrak{p}$.  
The form $\langle , \rangle_X$ is a $^*$-invariant form on $X$ and $\langle, \rangle_V$ is a $^\bullet$-invariant form on $V$. Let $\phi$, $\psi \in F_{\mu,k}(X) = \Hom_K(\mu, X\otimes V^{\otimes k})$, then

$$\langle \pi \tilde{z_i}(\phi), \psi\rangle_{F{\mu,k}(X)}= \langle\pi( \tilde{z_i})(\phi(1)), \psi(1) \rangle_{X \otimes V^{\otimes k}}.$$
Since $\pi(\tilde{z_i}) = \Omega_{0i}^\mathfrak{p} = \sum_{b \in B \cap \mathfrak{p} }(b)_0 \otimes (b^*)_i,$
$$\langle \pi \tilde{z_i}^*(\phi), \psi\rangle_{F{\mu,k}(X)}=  \langle(\Omega_{0i}^\mathfrak{p}) ^* \phi(1), \psi(1) \rangle_{X \otimes V^{\otimes k}},$$
$$=  -\langle\Omega_{0i}^\mathfrak{p} \phi(1), \psi(1) \rangle_{X \otimes V^{\otimes k}},$$
$$= -\left\langle\left( \sum_{b \in B \cap \mathfrak{p}} (b)_0 \otimes (b)_i\right) \phi(1),\psi(1)\right\rangle_{X\otimes V ^{\otimes k}}.$$
Denote $\phi(1)$ by $ \sum x_0 \otimes v_1 \otimes ...\otimes v_k$ and $\psi(1)$ by $\sum x_0' \otimes v_1' \otimes ...\otimes v_k' $ substituting in the definition of $\langle , \rangle_{X\otimes V^{\otimes k}}$ then 
$$\langle \pi \tilde{z_i}(\phi), \psi\rangle_{F{\mu,k}(X)}=\sum_{b \in B \cap \mathfrak{p}}\sum\langle -(b)_0(x_0),x_0'\rangle_X\langle v_1,v_1'\rangle_V ... \langle (b)_iv_i,v_i' \rangle_V...\langle v_k,v_k' \rangle_V.$$
The form $\langle , \rangle_X$ is $^*$-invariant for $\mathfrak{g}$ and $\langle, \rangle_V$ is $^\bullet$-invariant for $\mathfrak{g}$, hence $\langle -b x_0 , x_0'\rangle_X =\langle x_0 , b x_0'\rangle_X$ and $\langle b v_i ,v_i'\rangle_V =\langle v_i , b v_i'\rangle_X$ for all $b \in \mathfrak{p}$:
$$\langle \pi \tilde{z_i}(\phi), \psi\rangle_{F{\mu,k}(X)}=\sum_{b \in B \cap \mathfrak{p}}\sum\langle (x_0),(b)_0x_0'\rangle_X\langle v_1,v_1'\rangle_V ... \langle v_i,(b)_iv_i' \rangle_V...\langle v_k,v_k' \rangle_V.$$
Reversing through the definitions, we show 
$$ \langle \pi(\tilde{z_i})^*\phi, \psi \rangle_{F_{\mu,k}}= \langle \phi, \pi(\tilde{z_i})\psi \rangle_{F_{\mu,k}}.$$
The element $\theta_j$ acts by $(\xi)_j$  where $\xi \in \mathfrak{k}$, hence $\langle \xi v ,v' \rangle_V = \langle v,\xi v' \rangle_V$. Therefore
$$\langle x_0, x_0'\rangle_X \langle v_1, v_1'\rangle_V...\langle (\xi)v_j, v_j'\rangle_V ... \langle v_k , v_k' \rangle_V =  \langle x_0, x_0'\rangle_X \langle v_1, v_1'\rangle_V...\langle v_j,(\xi) v_j'\rangle_V ... \langle v_k , v_k' \rangle_V.$$
Similarly for $s_{i,i+1}$

$$\begin{array}{rl}
&\langle s_{i,i+1} (x_0 \otimes v_1 \otimes ...v_k ) , x_0'\otimes v_1' \otimes ... \otimes v_k' \rangle_{X\otimes V^{\otimes k}}\\[1ex]
  
=&\langle x_0 \otimes v_1 \otimes ...\otimes v_{i+1} \otimes v_{i} \otimes ... \otimes v_k  , x_0'\otimes v_1' \otimes ... \otimes v_k' \rangle_{X\otimes V^{\otimes k}}\\[1ex]
=&\langle x_0'\otimes v_1 \otimes ... \otimes v_k, x_0 \otimes v_1' \otimes ...\otimes v_{i+1}' \otimes v_{i}' \otimes ... \otimes v_k    \rangle_{X\otimes V^{\otimes k}}\\[1ex]
=& \langle x_0 \otimes v_1 \otimes ...v_k  , s_{i,i+1} ( x_0'\otimes v_1' \otimes ... \otimes v_k') \rangle_{X\otimes V^{\otimes k}}.\end{array}$$
  
The projection $e_{i,i+1}$ acts on elementary tensors $x_0 \otimes v_1 \otimes ...\otimes v_k$ by 
$$e_{i,i+1}(x_0 \otimes v_1 \otimes ...\otimes v_k) = m x_0 \otimes v_1 \otimes ...\otimes v_{i-1} \otimes \pr_{12}(v_i \otimes v_{i+1} ) \otimes ..,\otimes v_k.$$
Then 
$$\begin{array}{rl}
&\langle e_{i,i+1}(x_0 \otimes v_1 \otimes ...\otimes v_k), x_0' \otimes v_1'\otimes ...\otimes v_k'\rangle_{X \otimes V^{\otimes k}},\\[1ex]
=& m\langle x_0 \otimes v_1 \otimes ...\otimes v_{i-1} \otimes \pr_{12}(v_i \otimes v_{i+1} ) \otimes ...\otimes v_k,  x_0' \otimes v_1'\otimes ...\otimes v_k'\rangle_{X \otimes V^{\otimes k}},\\[1ex]
=& \langle x_0,x_0'\rangle_X \langle v_1,v_1'\rangle_V ...\langle  \pr_{12}(v_i \otimes v_{i+1} ) , v_i'\otimes v_{i+1}' \rangle_{V \otimes V} ... \langle v_k,v_k'\rangle_V.\end{array}$$
Using Lemma \ref{pronVV} ,
$$= \langle x_0 \otimes v_1\otimes ...\otimes v_k,  x_0' \otimes v_1' \otimes ...\otimes v_{i-1}' \otimes \pr_{12}(v_i' \otimes v_{i+1}' ) \otimes ..,\otimes v_k', \rangle_{X \otimes V^{\otimes k}},$$
Therefore 
$$\langle e_{i,i+1}(\phi), \psi\rangle_{F_{\mu, k}}= \langle \phi, e_{i,i+1}(\psi)\rangle_{F_{\mu,k}}.$$

The module $F_{\mu,k}(X)$ has induced Hermitian form $\langle,\rangle_{F_{\mu,k}}$ which is $^*$-invariant on the generators of $\mathfrak{B}_k^\theta$. Hence $\langle,\rangle_{F_{\mu,k}}$ is a $^*$-invariant form. If $X$ is unitary then $\langle,\rangle_{X \otimes V^{\otimes k}}$ is positive definite. Hence $\langle,\rangle_{F_{\mu,k}}$ is a positive definite invariant Hermitian form and $F_{\mu,k}(X)$ is unitary. 

\end{proof}

\begin{definition} Let $X \in \mathcal{HC}(G)$, $\mathfrak{B}_k^\theta$-mod or $\mathbb{H}_k(c)$-mod module with invariant Hermitian form $\langle , \rangle_X$. We define the Langlands quotient $\overline{X}$ to be: 
$$\overline{X}= X / \rad\langle , \rangle_X,$$
where $\rad\langle , \rangle$ is the radical of the form $\langle , \rangle$.
\end{definition}

\begin{lemma}\label{langlandsquotient} Let $X$ be in $\mathcal{HC}(G)$-mod with Hermitian invariant form $\langle , \rangle_X$ and Langlands quotient $\overline{X}$. The form $\langle , \rangle_{F_{\mu,k}}$ is the endowed hermitian form of $F_{\mu,k}(X)$ from Definition \ref{inducedform} then 
$$F_{\mu,k}(\overline{X}) = \overline{F_{\mu,k}(X)} =F_{\mu,k}(X)/ \rad\langle, \rangle_{F_{\mu,k}}.$$
\end{lemma} 

\begin{proof} 
One can construct an exact sequence:
$$0 \longrightarrow \rad\langle , \rangle_X \longrightarrow X \longrightarrow \overline{X} \longrightarrow 0.$$
Exactness of the functors $F_{\mu,k}$ is given by Lemma \ref{exactness}. Hence there is an exact sequence:

$$0 \longrightarrow F_{\mu,k}( \rad\langle , \rangle_X) \longrightarrow F_{\mu,k}( X) \longrightarrow F_{\mu,k}(\overline{X}) \longrightarrow 0.$$
For the result it is sufficient to prove that $F_{\mu,k}(\rad \langle, \rangle_X) = \rad \langle,\rangle_{F_{\mu,k}}$. Since $\langle ,\rangle_{F_{\mu,k}}$ is an invariant form on $F_{\mu,k}(X)$ and a non-degenerate form on $\overline{F_{\mu,k}(X)}$ then $F_{\mu,k}\rad\langle,\rangle_X = \rad\langle,\rangle_{F_{\mu,k}}$. 
\end{proof}

\langlandsquotientthm

\begin{definition} We define subsets of $\mathfrak{a}^*$:
 $$\mathcal{U}_{\delta} = \{\nu \in \mathfrak{a}^*: \overline{X_{\delta}^\nu}\text{ is a unitary Harish-Chandra module}\}.$$
Similarly define $$\mathcal{U}_k(\mathbbm{1}) = \{\lambda \in \mathfrak{a}_k^*: \overline{X_{\lambda}} \text{ is a unitary } \mathbb{H}_k(c_\mu)\text{ module}\}$$ 
and
$$\mathcal{U}_{n-k}(\mathbbm{1}) = \{\bar{\lambda} \in \bar{\mathfrak{a}}_{n-k}^*: \overline{X({\bar{\lambda}}}) \text{ is a unitary } \mathbb{H}_{n-k}(c_{\underline{\mu}})\text{ module}\}.$$
\end{definition} 
Since $\mathfrak{a} = \mathfrak{a}_k \oplus \bar{\mathfrak{a}}_{n-k}$ we can associate a pair $(\lambda_k,\bar{\lambda}_{n-k}) \in \mathfrak{a}^* \times \bar{\mathfrak{a}}_{n-k}^*$ to a character of $\mathfrak{a}$ via: 
$$(\lambda_k,\bar{\lambda}_{n-k}): \mathfrak{a} \to \mathbb{C}$$
$$(\lambda_k,\bar{\lambda}_{n-k})(\mathfrak{a}_k) = \lambda_k(\mathfrak{a}_k),$$
$$(\lambda_k,\bar{\lambda}_{n-k})(\bar{\mathfrak{a}}_{n-k}) = \bar{\lambda}_{n-k}(\bar{\mathfrak{a}}_{n-k}).$$
 Theorem \ref{langlandsquotientpreserved} shows that the Langlands quotients of $X_{\delta}^\nu$ map under $F_{\mu,k}$ and $F_{\underline{\mu},n-k}$ to Langlands quotients of principal series modules for $\mathbb{H}_k(c_\mu)$ and $\mathbb{H}_{n-k}(c_{\underline{\mu}})$ hence we can formulate the following non-unitary test.

\begin{lemma} We have an inclusion of sets:
$$\mathcal{U}_{\delta} \subseteq \mathcal{U}_k(\mathbbm{1}) \times \mathcal{U}_{n-k}(\mathbbm{1}) .$$\end{lemma}
This inclusion of sets states that if we take a minimal principal series module $X$ and find that, under the functor $F_{\mu,k}$, the Langlands quotient of the image is not unitary then the Langlands quotient of $X$ is not unitary.
\nonunitarytestthm

The above working does not require the image to be a Hecke algebra module. Therefore, we have also proved the following theorem. 
\nonunitarybrauerthm

We finish with a toy example.
\begin{example} Let $G = Sp_{2}(\mathbb{R}) \cong SL_2(\mathbb{R})$. 
Then principal series modules are labelled by $\delta = \pm 1$ and $\nu \in \mathbb{C}$. The principal series modules $X_\delta^\nu$ are unitary exactly when $\nu = i b$ for $b \in \mathbb{R}$, that is $\nu$ is entirely imaginary. 
 
In this case all principal series modules are spherical principal series modules. The root system associated to $Sp_2$ has one root $\epsilon$ and the Weyl group is $\mathbb{Z}_2$ which acts by $-1$ on $\epsilon$. Here $\mathbb{H}(c)$ will be the graded Hecke algebra associated to type $B_1$ with parameter $c$. The algebra $\mathbb{H}(c)$ is generated by $\epsilon$ and $s \in \mathbb{Z}_2$ such that $s \epsilon = - \epsilon s + c$. We note that $s^* = s$ and $\epsilon^* = -\epsilon + c s$.  

Our theorem gives that 
$$F_{\triv,1}(X_1^\nu) \cong \Ind_{\mathbb{C}}^{\mathbb{H}(c)} \mathbbm{1}_\nu.$$

Note that $\Ind_{\mathbb{C}}^{\mathbb{H}(c) }\mathbbm{1}_\nu$ is two dimensional with basis $\{\mathbbm{1}_\nu, s \mathbbm{1}_\nu\}$ we will denote the module $\Ind_{\mathbb{C}}^{\mathbb{H}(c)} \mathbbm{1}_\nu$ by $Y_\nu$. 
Let $\langle  , \rangle$ be a hermitian form on $Y_\nu$, for $Y_\nu$ to be unitary we require 
$$\langle s(u), v \rangle = \langle u, s^*(v)\rangle = \langle u,s(v)\rangle$$ and 
$$\langle \epsilon(u), v \rangle = \langle u ,\epsilon^*(v)\rangle = \langle u,[ -\epsilon + c s](v)\rangle.$$
Letting $u = \mathbbm{1}_\nu$ and $v = \mathbbm{1}_\nu$, then the above requirement implies 
$$\nu \langle \mathbbm{1}_\nu, \mathbbm{1}_\nu \rangle = \langle \epsilon(\mathbbm{1}_\nu) ,\mathbbm{1}_\nu\rangle =\langle \mathbbm{1}_\nu, [-\epsilon + c s]\mathbbm{1}_\nu\rangle = -\bar{\nu} \langle \mathbbm{1}_\nu, \mathbbm{1}_\nu \rangle + \langle \mathbbm{1}_\nu, s\mathbbm{1}_\nu\rangle.$$
For the above equation to hold $\nu =- \bar{\nu}$ and $\langle \mathbbm{1}_\nu,s\mathbbm{1}_\nu\rangle = 0$. Hence for $Y_\nu$ to be unitary $\nu$ must be purely imaginary. Furthermore if $\nu$ is purely imaginary then we can construct a Hermitian non-degenerate form on $Y_\nu$ such that it is a unitary form. Therefore in the case of $Sp_2(\mathbb{R})$ our non-unitary test becomes an equivalence: 
$$X_\delta^\nu \text{ is unitary if and only if } F_{\triv,1}(X_\delta^\nu) \cong \Ind_{\mathbb{C}}^{\mathbb{H}(c)} \mathbbm{1}_\nu \text{ is unitary.}$$
\end{example}

\end{section}

\addcontentsline{toc}{chapter}{Bibliography}
\bibliography{/home/k/google_drive/Dphil_work/bib.bib}        
\bibliographystyle{abbrv}  

\end{document}